\documentclass[11 pt, reqno]{amsart}

\usepackage[letterpaper, hmarginratio=1:1]{geometry}
\usepackage[hidelinks]{hyperref}
\usepackage{cleveref}
\usepackage[noadjust]{cite}
\usepackage{color}
\usepackage{bbm}
\usepackage{amssymb}
\usepackage{comment}

\theoremstyle{definition}
\newtheorem{ccounter}{ccounter}[section]
\newtheorem{theorem}[ccounter]{Theorem}

\newtheorem{lemma}[ccounter]{Lemma}
\newtheorem{corollary}[ccounter]{Corollary}

\newtheorem{proposition}[ccounter]{Proposition}
\newtheorem{definition}[ccounter]{Definition}
\newtheorem{remark}[ccounter]{Remark}
\numberwithin{equation}{section}

\DeclareMathOperator*{\argmin}{arg\,min}
\newcommand\C{\mathbb{C}}
\newcommand\E{\mathbb{E}}
\newcommand\N{\mathbb{N}}
\renewcommand\H{\mathbb{H}}
\renewcommand\P{\mathbb{P}}
\newcommand\R{\mathbb{R}}
\newcommand\be{\begin{equation}}
\newcommand\ee{\end{equation}}

\def\beb{\color{blue}}
\def\eeb{\normalcolor}
\def\Im{\operatorname{Im}}
\def\Re{\operatorname{Re}}
\def\one{{\mathbbm 1}}

\renewcommand{\tilde}{\widetilde}
\newcommand{\pf}{\operatorname{Pf}}
\newcommand{\iu}{\mathrm{i}}
\newcommand{\eps}{\varepsilon}
\newcommand{\bs}{\boldsymbol}
\renewcommand{\phi}{\varphi}

\title[Smallest gaps between eigenvalues of Gaussian matrices]{Smallest gaps between eigenvalues of Real Gaussian matrices}
\author{Patrick Lopatto}
\address{Division of Applied Mathematics\\Brown University\\Providence, Rhode Island 02906}
\email{patrick\_lopatto@brown.edu}
\author{Matthew Meeker}
\address{Division of Applied Mathematics\\Brown University\\Providence, Rhode Island 02906}
\email{matthew\_meeker@brown.edu}

\begin{document}
\begin{abstract}
We consider an $n\times n$ matrix of independent real Gaussian random variables and determine the asymptotic distribution of the smallest gaps between complex eigenvalues.
\end{abstract}
\maketitle

\section{Introduction}
\subsection{Background}\label{s:background}
Random matrices have fascinated mathematicians and physicists for decades due to their connections to 
quantum chaos, number theory, statistics, and numerous other fields. A primary focus is the distribution of gaps between consecutive eigenvalues in random Hermitian matrices, which underpins many of these links.  Substantial empirical evidence indicates that this distribution also arises in the spacings between energy levels of disordered quantum systems and the zeros of the Riemann zeta function, to give just two examples \cite{odlyzko1987distribution,bohigas1984characterization}.

While the distribution of a single eigenvalue gap is now mathematically understood for a wide variety of matrix models, less is known about the smallest and largest gaps. As motivation for their study, we note that the average-case performance of the Toda flow algorithm for diagonalizing a symmetric matrix can be analyzed in terms of the smallest eigenvalue gap of a Gaussian  matrix \cite{ben2013extreme,deift1983ordinary}. Additionally, the correspondence between eigenvalue gaps and spacings of zeta function zeros mentioned previously extends to the largest and smallest spacings in a given interval \cite{ben2013extreme}.

The rigorous study of extremal eigenvalue gaps was initiated by Vinson in 
his 2001 Ph.D.\ dissertation \cite{vinson2001closest}.
Using the method of moments, Vinson determined the asymptotic distribution of the smallest eigenvalue gap for the circular unitary ensemble (CUE), the Gaussian
unitary ensemble (GUE), and unitarily-invariant $\beta$-ensembles. 
For the CUE and GUE, these results were extended by Ben Arous and Bourgade in \cite{ben2013extreme}, where they obtained the joint limiting distributions for the $k$ 
smallest eigenvalue gaps. Instead of the method of moments, they drew on ideas developed by Soshnikov to study the smallest gaps of determinantal point processes \cite{soshnikov2005statistics}. Further, they also determined the asymptotic distribution of largest gaps for the CUE and GUE in the spectral bulk.
The smallest gaps distribution for the Gaussian orthogonal ensemble (GOE) was 
later established by Feng, Tian, and Wei in \cite{feng2019small}. We remark that
determining the asymptotic distribution of the largest gaps for the GOE remains 
an open problem.

All of the matrix models mentioned in the previous paragraph are exactly solvable, in the sense that 
their eigenvalue correlation functions are given by explicit formulas. There has also been significant interest in studying extremal eigenvalue gaps for matrices outside of this class.
In \cite{bourgade2021extreme}, Bourgade studied the distributions of 
smallest and largest gaps for a quite general set of random matrices, the generalized Wigner matrices, under a smoothness assumption on the entry distributions. He showed that these extremal gaps match those of the GOE/GUE (depending on whether the matrix is real symmetric or complex Hermitian). 
Landon, Lopatto, and Marcinek provided an alternative comparison argument for the largest gaps in \cite{landon2020comparison}, which does not require a smoothness hypothesis.
As a consequence of these universality results and the works on the GOE/GUE mentioned previously, the largest gaps distribution is known for all Hermitian generalized Wigner matrices, and the smallest gaps distribution is known for generalized Wigner matrices of both symmetry classes with sufficiently smooth entries. 

Further, lower bounds on the smallest gaps were established by Nguyen, Tao, and Vu in 
\cite{nguyen2017random} for Wigner matrices
with arbitrary mean (without a smoothness hypothesis), and for adjacency matrices of random graphs. Lopatto and Luh obtained similar lower bounds  in \cite{lopatto2021tail} for sparse matrices,
including adjacency matrices for sparse random graphs.
Feng and Wei studied the smallest gaps for the circular $\beta$-ensemble, a generalization
of the CUE, for all positive integer $\beta$ in \cite{feng2021small}.
In \cite{figalli2016universality}, Figalli and Guionnet extended the results of
Ben Arous and Bourgade from \cite{ben2013extreme} to a several-matrix model.

So far, we have discussed only models with one-dimensional spectra, with eigenvalues lying on either the real line or the unit circle.
In \cite{shi2012smallest}, Shi and Jiang studied
the distribution of the smallest gaps for the complex Ginibre ensemble, a matrix of independent complex Gaussian random variables, whose spectrum is asymptotically supported in the unit disk in the complex plane. To the best of our knowledge, this is the only previous
work identifying the asymptotic distribution of extremal gaps for an ensemble possessing
a two-dimensional spectrum. (Shi and Jiang also consider Wishart matrices and unitarily-invariant $\beta$-ensembles, which have one-dimensional spectra.)

In \cite{ge2017eigenvalue}, Ge proved a high-probability lower bound on the size of the smallest eigenvalue gap of any $n\times n$ random matrix with independent and identically distributed entries, subject to a mild regularity condition on the the entry distribution (that is satisfied, for example, by all distributions with finite variance). Luh and O'Rourke proved a stronger lower bound in \cite{luh2021eigenvectors} and deduced that when such a matrix has sub-Gaussian entries, the probability it does not have simple spectrum decays exponentially in $n$. 

High-probability lower bounds for the smallest particle gap in a two-dimensional Coulomb gas were proved by Ameur in \cite{ameur2018repulsion}, and by Ameur and Romero in \cite{ameur2023planar}. In \cite{thoma2022overcrowding}, Thoma studied the smallest gap for Coulomb gases in two and three dimensions, providing upper and lower bounds and proving asymptotic tightness. We remark that Thoma's lower bound in two dimensions improves on those in \cite{ameur2018repulsion, ameur2023planar}, and that he proves many other results on the separation of particles that hold in arbitrary dimension.

In this work, we consider the smallest gaps distribution for the real Ginibre ensemble, a matrix of independent real Gaussian random variables, whose spectrum is also asymptotically supported on the unit disk (but with a non-zero probability of real eigenvalues, unlike its complex counterpart). Due to the strong correlations between eigenvalues, we find that the smallest gaps between complex eigenvalues are of order $n^{-3/4}$. This contrasts with the smallest gap between $n$ independent points in the disk, which is of order $n^{-1}$, and confirms the general principle that random matrix eigenvalues act as mutually repelling particles whose interactions suppress their fluctuations (relative to those of independent particles).

The methodology of Shi and Jiang in \cite{shi2012smallest} was based on the fact that the eigenvalues of a complex Ginibre matrix form a determinantal point process, which makes the study of their smallest gaps amenable to the techniques developed in \cite{ben2013extreme, soshnikov2005statistics}. However, the eigenvalues of real Ginibre matrices instead have a Pfaffian structure, which makes understanding their smallest gaps substantially more complicated. This work is therefore the first to determine the smallest gaps distribution for a two-dimensional ensemble without determinantal correlation functions.



\subsection{Main Result}
We begin by introducing some concepts necessary to precisely state our main result. When possible, our notation and definitions are chosen to match the previous works \cite{shi2012smallest,goel2023central}, for consistency with the existing literature.

\begin{definition}
For all $n\in \N$, let $G_n = (g_{ij} )_{1\le i,j \le n}$ denote a $n\times n$ random matrix whose entries are independent Gaussian random variables such that $\E[g_{ij}] = 0$ and $\E[ g_{ij}^2] = 1$ for all $i,j$. 
The matrix $G_n$ is called the \emph{real Ginibre matrix} (GinOE) of dimension $n$. We also define $W_n = n^{-1/2} G_n$. 
\end{definition}

We refer the reader to \cite{byun2022progress}, \cite{byun2023progress}, and 
\cite{forrester2010log} for more information about the GinOE
and the related unitary and symplectic Ginibre ensembles.

Next, we recall two well-known facts about the spectrum of $W_n$. First, in the limit as $n$ tends to infinity, the empirical spectral distribution becomes uniformly distributed on the unit disk $\mathbb{D} = \{ z \in \mathbb{C} : |z| < 1 \}$ \cite{bordenave2012around}. Second, the non-real eigenvalues of $W_n$ come in conjugate pairs, since $W_n$ has real entries. This means that if $\lambda$ is an eigenvalue of $W_n$ with non-zero imaginary part, then $\bar \lambda$ is also an eigenvalue of $W_n$, and the eigenvalues in the upper half-plane $\mathbb{H} = \{ z \in \mathbb{C} :  \Im z > 0  \}$ completely determine those in the lower half-plane. 

Given these  considerations, we restrict our attention to the eigenvalues of $W_n$ lying in a given domain contained in the upper half disk \[\mathbb{D}^+ = \{ z \in \mathbb{C} : |z| < 1, \Im z > 0  \},\] and study the asymptotic distribution of the smallest gaps among these eigenvalues.\footnote{We recall that a domain is is a non-empty connected open subset of $\mathbb{C}$.} Further, to avoid boundary effects, we will consider only domains at a positive distance from the boundary of $\mathbb{D}^+$. This restriction leads to the following definition.
\begin{definition}
A domain $\Omega$ is called \emph{admissible} if   $\overline{\Omega} \subset \mathbb{D}^+$. 
\end{definition}

We will derive the distribution of the smallest gaps from a more general result about the convergence of a certain point process built from these gaps to a Poisson limit.
To define this process, we begin with the definition of an order on points of  $\C$. 
\begin{definition}
For $z_1, z_2 \in \mathbb{C}$, we say that $z_1 \prec z_2$ if $\Im(z_1) < \Im(z_2)$, or if $\Im(z_1) = \Im(z_2)$ and $\Re(z_1) < \Re(z_2)$.
\end{definition}
Let $\{\lambda_i\}_{i=1}^n$ denote the eigenvalues of $W_n$, indexed so that $\lambda_1 \prec \dots \prec \lambda_n$ when all eigenvalues are distinct. On the measure zero set where the eigenvalues are not distinct, we label the eigenvalues in the same way, except we break ties between equal eigenvalues arbitrarily.

\begin{definition}
Let $\Omega$ denote an admissible domain and set $\R^+ = [0,\infty)$. 
We define a point process $\chi^{(n)}_\Omega$ on $\R^+$ as follows. First, for all $i \in \N$ such that $1 \le i \le n$, we define 
\begin{equation}
i^* = \argmin_{ j\neq i} \big\{
| \lambda_j - \lambda_i| : \lambda_i \prec \lambda_j , \lambda_j \in \Omega
\big\}
\end{equation}
if the set of indices $j$ such that $j\neq i$ and $\lambda_j \in \Omega$ is nonempty. Otherwise, we set $i^* = 0$. We then set
\begin{equation}
\chi_{\Omega}^{(n)} = \sum_{i : \lambda_i \in \Omega} \delta_{n^{3/4} | \lambda_{i^*} - \lambda_i|} \one_{ \{ {i^*} \neq 0 \}}.
\end{equation}
\end{definition}
\begin{remark}
As noted in \cite[Remark 1.1]{shi2012smallest}, the point of introducing the order $\prec$ in the definition of $i^*$ is to prevent the duplication of gaps in $\chi_{\Omega}^{(n)}$, to ensure good limiting behavior. For example, if $i^*$ were defined as the index minimizing $|\lambda_i - \lambda_{i^*}|$ (without the order condition), then the smallest gap would appear twice in the set $\{ | \lambda_i - \lambda_{i^*} | :  1\le i \le n\}$, and $\chi_{\Omega}^{(n)}$ could not converge to a Poisson process with an absolutely continuous intensity measure.
\end{remark}

We now state our main theorem on the Poisson convergence of $\chi_{\Omega}^{(n)}$.
For every set $S \subset \C$, let $|S|$ denote the Lebesgue measure of $C$. 

\begin{theorem}\label{t:main}
Let $\Omega$ be an admissible domain.
As $n\rightarrow \infty$, the processes $\chi^{(n)}_{\Omega}$ converge weakly to a Poisson point process $\chi_{\Omega}$ on $\R^+ $ with intensity
\begin{equation}\label{lastclaim}
\E\big[ \chi_\Omega( A)\big]
= \frac{ | \Omega |}{\pi} \int_A r^3\, dr
\end{equation}
for any bounded Borel set $A \subset \R^+$.
\end{theorem}

For the next corollary, we let 
\be 
t_1^{(n)} = \min \{ | \lambda_i  - \lambda_j| :  \lambda_i \in \Omega \land \lambda_j \in \Omega \land i < j   \}
\ee
denote the smallest gap between eigenvalues in $\Omega$, and more generally we let $t_\ell^{(n)}$ denote the $\ell$-th smallest value in this set. We then define the rescaled gaps
\be 
\omega^{(n)}_\ell = n^{3/4} \left(\frac{\pi }{4 | \Omega| }\right)^{1/4} t_\ell^{(n)}.
\ee
\begin{corollary}\label{c:maincor}
For any real numbers $0< x_1 < y_1 < \dots < x_k < y_k$, we have 
\[
\lim_{n\rightarrow \infty}
\P(x_\ell <  \omega^{(n)}_{\ell} < y_\ell \text{ for all } 1 \le \ell \le k )
 = \left( e^{-x^4_k} - e^{-y_k^4} \right) \prod_{\ell=1}^{k-1} ( y_\ell^4 - x_\ell^4).
\]
\end{corollary}
In particular, the previous result implies that for any fixed $k\in \mathbb{N}$, the rescaled gap $\omega^{(n)}_{k}$ converges in distribution to a random variable with density $p(x) \propto x^{4k-1} e^{-x^4}\, dx$.

\begin{remark}
The limiting intensity measure in \eqref{lastclaim} is identical to the one for the smallest gaps process of the complex Ginibre ensemble (GinUE); see \cite[Theorem 1.1]{shi2012smallest}.\footnote{Consequently, the distributions of the rescaled smallest gaps also match.  We note for the convenience of the reader that there is a misprint in the definition of the rescaled gaps $\tau^{(n)}_\ell$ in \cite[Corollary 1.1]{{shi2012smallest}}.}
\end{remark}

\subsection{Proof Ideas}

As mentioned in \Cref{s:background}, the previous analysis of smallest gaps for the GinUE in
\cite{shi2012smallest} is based on the determinantal structure of the eigenvalue correlation functions, following ideas of \cite{soshnikov2005statistics} and \cite{ben2013extreme}.
Because the GinOE has Pfaffian correlation functions, this
methodology does not immediately apply.
However, it was observed in \cite{goel2023central}
and \cite{kopel2015linear} that in the bulk of the spectrum, these Pfaffian correlation 
functions can be approximated by determinantal ones up to exponentially small 
additive errors.
Using this idea, we are able to place ourselves back in the determinantal framework of \cite{shi2012smallest, soshnikov2005statistics, ben2013extreme}.

After we make the reduction to a determinantal process, we come to the primary obstacle in adapting the strategy used for the smallest gaps of the GinUE in \cite{shi2012smallest}. Their proof relies heavily on the positive-definiteness of the determinantal kernel for the GinUE, while the kernel in our determinantal approximation is not known to be positive-definite. 
To overcome this difficulty, we prove a technical
lemma showing that the approximating kernel is positive-definite (or exponentially small) everywhere in the spectral bulk, except possibly for a set of asymptotically vanishing measure (see \Cref{l:integration}). Using the determinantal approximation in tandem with
this lemma, we are able to complete the proof of our main result by carefully tracking the contribution of the exceptional sets on which the kernel is not positive definite.

It is natural to wonder whether \Cref{t:main} holds for more general domains. Indeed, the analysis of the smallest gaps for the GinUE in \cite{shi2012smallest} holds for any region of the complex plane. The methods here should extend in a relatively straightforward way to any region $\Omega$ such that $\overline{\Omega} \subset \mathbb{H}$; that is, we require a positive distance from the real axis but permit the region to extend beyond the unit circle. Specifically, one can augment the asymptotic analysis of the Pfaffian kernel in \Cref{l:23}, which holds in the interior of the unit disk, with the asymptotics from \cite[Remark 3.4]{kriecherbauer2008locating} and \cite[Theorem B.1]{bleher2006zeros} to access to the entire interior of the upper half-plane. We omit this extension for brevity, since it requires lengthy computations. Generalizing our result to regions that intersect the real line would require a precise accounting of the contributions from the real eigenvalues, and seems more difficult. 

\subsection{Outline}
In \Cref{s:preliminaryresults}, we begin by recalling previous results on the correlation
functions for the GinOE, the Pfaffian, and the determinantal approximation, and state some straightforward consequences of these results. We also recall some facts about convergence to Poisson distributions and processes.
In \Cref{s:main},
we state a three key lemmas and show how they imply our main theorem and its corollary. 
In \Cref{s:proof_main_lemmas}, we prove each of these key lemmas. 
Finally, in \Cref{s:auxiliary}, we prove several auxiliary technical lemmas that we
require at various points throughout this work.

\subsection{Acknowledgments}

P.\ L.\ is supported by NSF postdoctoral fellowship DMS-2202891.

\section{Preliminary Results}\label{s:preliminaryresults}

\subsection{Correlation Functions for the GinOE}
We begin by recalling how to compute (symmetrized) expectations of functions of the eigenvalues $(w_i)_{i=1}^n$  of $G_n$.
Let $\C^* = \C {\setminus} \R$ denote the complex plane with the real line removed, and let $\mathcal{I}_k\subset \{1, \dots , n\}^n$ be the set of pairwise distinct $k$-tuples of indices. 
By \cite[(5.1)]{borodin2009ginibre}, for all $k, n \in \mathbb{N}$, there exists a function  $\rho_k^{(n)}\colon (\C^*)^k \rightarrow \R^+$ such that 
\begin{equation}
\int_{(\C^*)^k } f(z_1,\dots, z_k) 
\rho_k^{(n)} (z_1, \dots, z_k) \,  dz_1\dots dz_k
 = 
 \E 
 \left[
 \sum_{(i_1, \dots, i_k) \in \mathcal{I}_k } f(w_{i_1},\dots, w_{i_k})
 \right]\label{corrdef}
\end{equation}
for every compactly supported, bounded, and 
Borel-measurable function $f\colon (\C^*)^k \rightarrow \R$. We will use the shorthand $\rho_k = \rho_k^{(n)}$, suppressing the $n$-dependence in the notation.
 For more on correlation functions, including their definition for general random point fields, we refer the reader to 
\cite{soshnikov2000determinantal}.
\eeb

The next lemma identifies the correlation function $\rho_k$ as the Pfaffian of certain $2k \times 2k$ matrix. For the reader's convenience, we begin by recalling the definition of a Pfaffian.
\begin{definition}
The Pfaffian of a $2n \times 2n$ skew-symmetric matrix $M=(M_{ij})_{i,j=1}^{2n}$ is defined by
\begin{equation*}
    \pf(M)=\frac{1}{2^n n !} \sum_{\sigma \in S_{2 n}} \operatorname{sgn}(\sigma) \prod_{i=1}^n M_{\sigma(2 i-1), \sigma(2 i)},
\end{equation*}
where $S_{2n}$ is the symmetric group of degree $2n$. 
\end{definition}

The statement below quotes \cite[Theorem 1.1]{goel2023central} exactly, which collected certain results from \cite[Appendix B.3]{mays2012geometrical}. However, we emphasize that these correlation functions were originally identified explicitly in \cite{forrester2007eigenvalue}, and the Pfaffian form below was first derived in the works \cite{borodin2009ginibre,forrester2009method,sinclair2009correlation,sommers2008general}.

\begin{lemma}\label{l:21}
The $k$-point complex--complex correlation functions of the $n$-dimensional real Ginibre ensemble $G_n$ are given by
\be\label{pfcorr}
    \rho_k(z_1, \ldots, z_k)=\pf(K(z_i, z_j))_{1 \leq i, j \leq k},
\ee
where $(K(z_i, z_j))_{1 \leq i, j \leq k}$ is a $2k \times 2k$ matrix composed of the $2\times 2$ blocks
\begin{align*}
    K(z_i, z_j)=\begin{pmatrix}
D_{n}\left(z_i, z_j\right) & S_{n}\left(z_i, z_j\right) \\
-S_{n}\left(z_j, z_i\right) & I_{n}\left(z_i, z_j\right)
\end{pmatrix},
\end{align*}
and $D_n$, $I_n$, and $S_n$ are defined by
\begin{gather*}
S_{n}(z, w)= \frac{\mathrm i e^{-(1 / 2)(z-\bar{w})^2}}{\sqrt{2 \pi}}(\bar{w}-z) G(z, w) s_{n}(z \bar{w}), \\
D_{n}(z, w) =\frac{e^{-(1 / 2)(z-w)^2}}{\sqrt{2 \pi}}(w-z) G(z, w) s_{n}(z w), \\
I_{n}(z, w) =\frac{e^{-(1 / 2)(\bar{z}-\bar{w})^2}}{\sqrt{2 \pi}}(\bar{z}-\bar{w}) G(z, w) s_{n}(\bar{z} \bar{w}),
\end{gather*}
where $z,w \in \C^*$ and 
\begin{gather*}
G(z, w)=\sqrt{\operatorname{erfc}(\sqrt{2} \Im(z)) \operatorname{erfc}(\sqrt{2} \Im(w))},\qquad \operatorname{erfc}(x) = \frac{2}{\sqrt{\pi}} \int_x^\infty \exp( - t^2)\, dt,\\
s_{n}(z)=e^{-z} \sum_{j=0}^{n-1} \frac{z^j}{j !}.
\end{gather*}
\end{lemma}

\begin{remark}\label{r:23}
By a change of variables in the definition \eqref{corrdef}, it follows that the $k$-th correlation function for the complex eigenvalues of $W_n$ is $n^{k} \rho_k(\sqrt{n} z_1, \dots \sqrt{n} z_k)$.
\end{remark}
The following lemma is useful for controlling the correlation functions of $W_n$. Its statement is taken from \cite[Lemma 2.3]{goel2023central}. 

\begin{lemma}\label{l:23}
Let $\Omega$ be an admissible domain, and let $d_\Omega = \inf \{ |z-w| : z \in \Omega, w\in \partial  \mathbb{D}^+ \}$ denote the distance between $\Omega$ and the boundary of $\mathbb{D}^+$. There exist  constants  $C(d_\Omega), c(d_\Omega)>0$ such that  
\begin{gather*}
\sup_{z,w \in \Omega} \big| D_n(\sqrt{n} z,\sqrt{n}w)\big|   \le C e^{-cn},
\qquad \sup_{z,w \in \Omega} \big| I_n(\sqrt{n}z,\sqrt{n}w)\big| \le C e^{-cn},\\
\sup_{z,w \in \Omega} \big| S_n(\sqrt{n}z,\sqrt{n} w)\big|  \le C.
\end{gather*}
\end{lemma}

The statement of the next lemma is from \cite[Lemma 2.4]{goel2023central}. It was proved in \cite[Appendix B]{gebert2019pure} (and appeared earlier in \cite{forrester2008skew}).
\begin{lemma}\label{l:24}
Let $M=(M_{ij})_{i,j=1}^{2n}$ be a skew-symmetric $2n \times 2n$ matrix such that $M_{ij} = 0$ for every pair of indices $(i,j)$ such that $i \equiv j \bmod 2$. Let $\tilde M = (\tilde M)_{i,j=1}^n $ be the $n\times n$ matrix defined by $\tilde M_{ij} = M_{2i-1, 2j}$. 
Then $\pf(M) = \det (\tilde M)$.
\end{lemma}
For all $k\in \mathbb{N}$, define the $k\times k$ matrix $Q^{(k)}(z_1,\dots , z_k) =(S_n( \sqrt{n} z_i, \sqrt{n} z_j))_{1 \leq i, j \leq k}$. The following lemma provides a useful approximation of the correlation functions of $W_n$ by a determinant (see \Cref{r:23}). It follows from combining \Cref{l:21},  \Cref{l:23}, and \Cref{l:24}; we provide a detailed proof in \Cref{s:auxiliary}.

\begin{lemma}\label{l:detapprox}
Fix $k\in \N$. 
For every admissible domain $\Omega$, there exists a constant $c(k,d_{\Omega}) > 0$ such that for all $n\in \N$,
\begin{equation*}
\sup_{z_1,\dots, z_k \in \Omega}
\big|
n^k \rho_k(\sqrt{n} z_1, \dots, \sqrt{n} z_k)
-
n^k \det Q^{(k)}(z_1, \dots, z_k)
\big|
\le c^{-1} e^{-cn}.
\end{equation*}
\end{lemma}

The next lemma is a consequence of \Cref{l:detapprox} and provides a useful estimate on $S_n(\sqrt{n} z, \sqrt{n} w)$ when $z$ and $w$ are close together. Its proof also appears in \Cref{s:auxiliary}. We define 
\begin{equation*}
U(z,w) = \frac{\iu e^{(-n/2)(z - \bar w)^2 -n (\Im(z)^2 + \Im(w)^2)} }{2\pi \sqrt{  \Im(z) \Im(w)}} (\bar w - z).
\end{equation*}
\begin{lemma}\label{l:close}
Fix an admissible domain $\Omega$ and $r \in \R^+$. 
Then there exists a constant $C(r,d_\Omega) >0$ such that 
\begin{equation}
\sup_{z,w \in \Omega : |z -w | < r n^{-3/4}}
\left|
S_n(\sqrt{n} z, \sqrt{n} w) -
U(z,w)
\right| \le \frac{C}{n}.
\end{equation}
\end{lemma}

\subsection{Poisson Convergence}

We require two basic convergence results, one for Poisson random variables and one for Poisson point processes, which are proved in \Cref{s:auxiliary}.
 
\begin{proposition}\label{p:poisson}
Let $\{X_1 \}_{n=1}^\infty$ be a sequence of random variables taking values in the non-negative integers, and let $X$ be a Poisson random variable with rate $\lambda > 0$. Suppose that for all $k > 0$, 
\be\label{factorialmoments}
\lim_{n \rightarrow \infty}
\E 
\big[
X ( X - 1 ) \cdots (X - k) 
\big]
= \lambda^k.
\ee
Then the sequence $X_n$ converges to $X$ in distribution.
\end{proposition}

\begin{proposition}\label{p:bab}
Let $\{\chi^{(n)}\}_{n=1}^\infty$ be a sequence of point processes on $\R$, and let $\chi$ be a Poisson point process on $\R$ with a intensity measure $\mu$, which we suppose has no atoms. If $\chi^{(n)}(J)$ converges in distribution to $\chi(J)$ for all bounded Borel sets $J\subset \R$, then  the sequence of point processes $\chi^{(n)}$ converges in distribution to $\chi$.
\end{proposition}

\section{Proof of Main Result}\label{s:main}

We now state a series a lemmas and show how they imply the main result. Their proofs  follow in the coming sections. 

\subsection{Modified Point Process}
We begin by introducing a auxiliary point process, the \emph{s-modified} point process corresponding to the eigenvalues of $W_n$. This technique was originally introduced in the works  \cite{soshnikov1998level,soshnikov2005statistics}.

Given a  bounded Borel set $A \subset \R^+$, we define the corresponding set 
\begin{equation}\label{d:BA}
B= \{ z \in \C : | z| \in A, 0 \prec z \},
\end{equation}
omitting the dependence on $A$ in the notation (since the choice of $A$ will  always be clear from context). 
For all $n \in \N$, we set
\be
A_n = \{n^{-3/4} a : a \in A\},
\qquad
B_n =
\{n^{-3/4} b : b \in B\},
\ee
and define
\be
\xi^{(n)}
=\sum_{i  : \lambda_i \in \Omega} \delta_{\lambda_i},
\quad
\tilde \xi^{(n)}
=\sum_{i  : \lambda_i \in \Omega } \delta_{\lambda_i} \one_{\{\xi^{(n)}(\lambda_i + B_n) =1   \} },
\ee
where the set $\lambda_i + B_n$ has the usual definition as
$\lambda_i + B_n = \{ \lambda_i + z : z \in B_n\}$. 

The following lemma implies that it suffices to study $\tilde \xi^{(n)}$ in order to prove our main result. It is proved in \Cref{s:proof_main_lemmas}.
\begin{lemma}\label{l:33}
Let $\Omega$ be an admissible domain. 
For all bounded Borel sets $A \subset \R^+$, we have 
\begin{equation}\label{l33conclusion}
\lim_{n\rightarrow \infty} \chi^{(n)}(A)
- \tilde \xi^{(n)}(\Omega) =0,
\end{equation}
where the convergence is in distribution.
\end{lemma}

We let $\tau_k$ denote the $k$-th correlation function for $\xi^{(n)}$, suppressing the dependence on $n$ in the notation. From the definition of a correlation function (see e.g.\ \cite[Definition 2]{soshnikov2000determinantal}) and \Cref{r:23}, it is straightforward to see that $\tau_k$ is given by
\begin{equation}\label{taukdef}
\tau_k(z_1, \dots, z_k) =
n^k \rho_k(\sqrt{n} z_1, \dots, \sqrt{n} z_k) \one_{\Omega}(z_1) \dots \one_{ \Omega}(z_k).
\end{equation}
Further, we let $\tilde \tau_k$ denote the $k$-th correlation function for the modified process $\tilde \xi_n$. The function $\tilde \tau_k$ can be written explicitly in terms of the functions $\{\tau_{i}\}_{i=1}^{n}$ using the inclusion--exclusion principle; see \cite[(4.5)]{soshnikov2000determinantal} for details.

For every $k\in \N$, we define the set
\begin{equation}
\Psi_k
= 
\big\{
(z_1, \dots, z_k) \in \Omega^k : (z_i + B_n) \cap (z_j + B_n) = \varnothing \text{ for all } i\neq j, 1 \le i,j \le k
\big \}.
\end{equation}
The following lemma is proved in \Cref{s:proof_main_lemmas}.
\begin{lemma}\label{l:integration}
Fix an admissible domain $\Omega$, a bounded Borel set $A\subset \R^+$, and $k \in \N$. 
\begin{enumerate}
\item 
There exists a set $\mathcal Z \subset \Omega^k$ such that $\mu(\mathcal Z) =0$ and the following holds. For all $(z_1, \dots, z_k) \in \Omega^k \setminus \mathcal Z$ with pairwise distinct entries, 
\begin{equation}
\lim_{n \rightarrow \infty} 
\tilde \tau_k(z_1, \dots, z_k) = \left( \frac{1}{\pi^2} \int_{B} |z|^2 \, dz  \right)^k.
\end{equation}
\item There exist constants $C, c>0$, depending only on $A$, $k$, and $\Omega$, such that the following statements hold for all $n \in \N$. There exists a set $\mathcal W_n \subset \Omega^k$  such that
\be\label{W1}
\P ( \mathcal W_n ) \le C e^{-c n },
\ee
and for all $(z_1, \dots, z_k) \in \Psi_k \setminus \mathcal W_n$, we have 
\begin{equation}\label{W2}
\tilde \tau_k (z_1, \dots, z_k) \le C,
\end{equation}
and for all $(z_1, \dots, z_k) \in \mathcal W_n$, 
\be\label{W3}
\tilde \tau_k (z_1, \dots, z_k) \le C n^{8k}. 
\ee 
\item Set $\overline{\Psi}_k = \Omega^k \setminus \Psi_k$. Then
\begin{equation}
\lim_{n\rightarrow \infty}
\int_{\overline{\Psi}_k}
\tilde \tau_k(z_1, \dots, z_k)
\, dz_1 \dots d z_k = 0.
\end{equation}
\end{enumerate}
\end{lemma}

\begin{proof}[Proof of \Cref{t:main}]
By the definition of a correlation function \cite[Definition 2]{soshnikov2000determinantal},  we have 
\begin{equation}\label{descending}
\int_{\Omega^k} \tilde \tau_k( z_1, \dots ,  z_k) \, dz_1 \dots d z_k =
\E\left[ 
\frac{
\tilde \xi^{(n)}(\Omega)!}{
\big(  
\tilde \xi^{(n)}(\Omega) - k
\big)!
} \right]. 
\end{equation}
Abbreviating $\mathcal W = \mathcal W_n$, we write 
\begin{align}
\int_{\Omega^k} 
\tilde \tau_k(z_1, \dots, z_k)\, dz_1 \dots z_k
 =&
\int_{\Omega^k } 
 \one_{\Psi_k \setminus \mathcal W} \tilde \tau_k(z_1, \dots, z_k)\, dz_1 \dots z_k\label{rhs1} \\
 &+  \int_{\Omega^k} \one_{\mathcal W} \tilde \tau_k(z_1, \dots, z_k)\, dz_1 \dots z_k \label{rhs1b} \\ 
&+ \int_{\Omega^k} 
 \one _{\overline{\Psi}_k} \tilde \tau_k(z_1, \dots, z_k)\, dz_1 \dots z_k.\label{rhs2}
\end{align}
Observe that by the Borel--Cantelli lemma and \eqref{W1}, we have 
\be\label{patch}
\lim_{n\rightarrow \infty}
\one_{\Psi_k \setminus \mathcal W } = \one_{\Omega}.
\ee 
The integrand in the integral on the right-hand side of \eqref{rhs1} is uniformly bounded, by \Cref{l:integration}(1), so by the dominated convergence theorem, \Cref{l:integration}(2), and \eqref{patch}, we conclude that
\begin{equation}
\lim_{n\rightarrow \infty} \int_{\Omega^k } 
 \one_{\Psi_k} \tilde \tau_k(z_1, \dots, z_k)\, dz_1 \dots z_k
  =\left( \frac{1}{\pi^2} \int_{B} |z|^2 \, dz  \right)^k\left( \int_{\Omega} dz \right)^k.\label{rhs3}
\end{equation}
Using \eqref{W1} and \eqref{W3}, we see that 
\begin{equation}\label{red}
\lim_{n\rightarrow \infty} \int_{\Omega^k} \one_{\mathcal W} \tilde \tau_k(z_1, \dots, z_k)\, dz_1 \dots z_k  = 0.
\end{equation}
Further, by \eqref{rhs2}, the limit as $n$ tends to infinity of \eqref{rhs2} also vanishes. Then combining \eqref{descending}, \eqref{rhs1}, \eqref{red}, and \eqref{rhs3}, we find 
\begin{equation}
\lim_{n \rightarrow \infty} \E\left[ 
\frac{
\tilde \xi^{(n)}(\Omega)!}{
\big(  
\tilde \xi^{(n)}(\Omega) - k
\big)!
} \right]
= \left(  \int_{B} |z|^2 \, dz  \right)^k \left( \frac{1}{\pi^2}\int_{\Omega} dz \right)^k.
\end{equation}
It follows from \Cref{p:poisson} that $\tilde \xi^{(n)}(\Omega)$ converges in distribution to a Poisson random variable with intensity equal to 
\begin{equation}\label{poissonrate}
\left(\int_B |z|^2 \, dz \right)
\left( \int_{\Omega} \frac{dz}{\pi^2} \right) 
= \frac{ | \Omega |}{\pi} \int_A r^3\, dr.
\end{equation}
Then \Cref{l:33} implies that $\chi^{(n)}_\Omega( A)$ converges in distribution to a Poisson random variable with rate given by \eqref{poissonrate}. Since this convergence holds for all bounded Borel sets $A \subset \R^+$, \Cref{p:bab} implies that $\chi^{(n)}_\Omega$ converges to $\chi_\Omega$, as desired.
\end{proof}

For the proof of \Cref{c:maincor}, we additionally require the next lemma, which says that it is unlikely that three eigenvalues all bunch together on the scale of the smallest eigenvalue gap. 
It is also proved in \Cref{s:proof_main_lemmas}. 
We define the random measure
\be \label{xi}
\Xi^{(3)} = \sum_{\lambda_{i_1}, \lambda_{i_2}, \lambda_{i_3}\text{ pairwise distinct}} \delta_{(\lambda_{i_1}, \lambda_{i_2}, \lambda_{i_3})},
\ee 
on $\C^3$, 
and for all admissible domains $\Omega$ and all $M>0$, we define
\be
\mathcal B_M = \{
(z, x_1, x_2) : z \in \Omega, |x_1 - z| \le M n^{-3/4}, |x_2 - z| \le M n^{-3/4} \}.
\ee

\begin{lemma}\label{l:bunching}
Fix $M>0$ and an admissible domain $\Omega$. Then
\be
\lim_{n\rightarrow \infty} 
\E\big[ \Xi^{(3)} (\mathcal B_M)\big] = 0.
\ee
\end{lemma}

\begin{proof}[Proof of \Cref{c:maincor}]
Given \Cref{t:main}, the proof of this corollary is essentially the same as that of \cite[Corollary 1.1]{shi2012smallest}. We include it here for completeness.

Let $\tilde t_1 \le \dots \le \tilde t_k$ denote the $k$ smallest values in the set 
\be
\big\{ | \lambda_i  - \lambda_{i^*} | :  \lambda_i \in \Omega \land i^* \neq 0   \big\},
\ee 
and define the rescaled values 
\be 
\tilde \omega_\ell = n^{3/4} \left(\frac{\pi }{4 | \Omega| }\right)^{1/4} \tilde t_\ell.
\ee 
Observe that we have suppressed the $n$-dependence in the notation for the $\tilde t_k$ and $\tilde \omega_k$. We also use the shorthand $t_k = t^{(n)}_k$.
We will first show that the $\tilde \omega_\ell$ have the desired joint limiting distribution, then transfer this result to the $\omega_\ell$. 

We begin by observing that the event 
\be\label{tildeevent}
\{
 x_\ell
< \tilde \omega_\ell <  y_\ell \text{ for all } 1 \le \ell \le k \}
\ee
can be written as the intersection of the events
\be
\left\{
\chi^{(n)}_{\Omega}
\left( 
\left( \frac{ 4 | \Omega|}{\pi} \right)^{1/4} (x_\ell, y_\ell)
\right) \ge 1 
\right\},
\ee
\be
\left\{
\chi_{\Omega}
\left( 
\left( \frac{ 4 | \Omega|}{\pi} \right)^{1/4} (x_k, y_k)
\right) = 1 \text{ for all } 1\le \ell \le k -1
\right\},
\ee
and 
\be
\left\{
\chi_{\Omega}
\left( 
\left( \frac{ 4 | \Omega|}{\pi} \right)^{1/4} (y_{\ell-1}, x_\ell)
\right) = 0 \text{ for all } 1\le \ell \le k
\right\},
\ee
with the convention that $y_0 = 0 $. Using this representation, and the fact that $\chi_\Omega(S_1)$ and  $\chi_{\Omega}(S_2)$ are independent when the sets $S_1$ and $S_2$ are disjoint (since $\chi_{\Omega}$ is a Poisson process), we have 
\begin{multline}\label{tdist}
\lim_{n\rightarrow \infty}
\P(x_\ell <  \tilde \omega_{\ell} < y_\ell \text{ for all } 1 \le \ell \le k )=\\
\left(
1 - e^{-(y_k^4 - x_k^4)}
\right)
\prod_{\ell=1}^{k-1} ( y_\ell^4 - x_\ell^4) e^{-(y_\ell^4 - x_\ell^4)}  
\prod_{\ell=1}^k e^{(y_\ell^4 - x_\ell^4)}
 =\left( e^{-x^4_k} - e^{-y_k^4} \right) \prod_{\ell=1}^{k-1} ( y_\ell^4 - x_\ell^4).
\end{multline}

Next, we prove that 
\be\label{tildereplace}
\lim_{n \rightarrow \infty}
\P
(\exists \, \ell \le k \text{ such that } \tilde t_\ell \neq t_\ell ) =0.
\ee
Let $p_\ell < q_\ell$ denote the indices such that $| \lambda_{p_\ell} - \lambda_{q_\ell} | = t_\ell$. To show \eqref{tildereplace}, it suffices to show that the probability that the $2k$ eigenvalues $( \lambda_{p_\ell}, \lambda_{q_\ell})_{\ell = 1}^k$ are distinct tends to $1$ as $n\rightarrow \infty$.
Using the notation defined before \Cref{l:bunching}, we find that for every $M > 0$, 
\begin{align}
\P
(\exists \, \ell \le k \text{ such that } \tilde t_\ell \neq t_\ell ) &\le
\P \big(\Xi^{(3)} (\mathcal B_M) \neq 0 \big)
+ 
\P ( \tilde t_k > M n^{-3/4} /2 ) \\
& \le 
\E \big[\Xi^{(3)} (\mathcal B_M) \big]
+ \P ( \tilde t_k > M n^{-3/4} /2 ).
\end{align}
By \Cref{l:bunching}, 
\be
\lim_{n\rightarrow \infty} 
\E\big[ \Xi^{(3)} (\mathcal B_M)\big] = 0.
\ee
Therefore 
\be\label{M1}
\limsup_{n \rightarrow \infty}
\P
(\exists \, \ell \le k \text{ such that } \tilde t_\ell \neq t_\ell ) \le \limsup_{n \rightarrow \infty} \P ( \tilde t_k > M n^{-3/4} /2 ),
\ee
and we conclude the left-hand side of \eqref{M1} is zero by taking $M$ to infinity and using \eqref{tdist}.
\end{proof}

\section{Proofs of Main Lemmas}\label{s:proof_main_lemmas}

We begin with the proof of \Cref{l:bunching}, since it is used in the proof of \Cref{l:33}. We then prove \Cref{l:33} and \Cref{l:integration}. We use the convention that the letters $C$ and $c$ denote positive constants that may change line to line. 

\subsection{Proof of Lemma~\ref{l:bunching}}

\begin{proof}[Proof of \Cref{l:bunching}]
For all $z\in \C$ and $r\in \R^+$, let
\[
D(z,r) = \{ w \in \C : |z - w | < r\}.
\]
Set $M_n = M n^{-3/4}$. 
We first note that, by the identification of the correlation functions for $W_n$ in \Cref{r:23}, and \Cref{l:detapprox}, we have 
\begin{align}
\E \big[ \Xi^{(3)}(\mathcal B_M) \big]
&= n^3 \int_{\Omega} \int_{D(\lambda, M_n)^2}
\rho_3(\sqrt{n} \lambda, \sqrt{n} z_1 , \sqrt{n} z_1) \, dz_1 \, dz_2\, d\lambda\\
&= n^3 \int_{\Omega} \int_{D(\lambda, M_n)^2}
\det Q(\lambda, z_1, z_2) \, dz_1 \, dz_2\, d\lambda + O(e^{-cn}), \label{recall0}
\end{align}
where we abbreviate $Q(w_1, w_2, w_3) = Q^{(3)}(w_1, w_2, w_3)$. 
Here, and for the rest of this proof, the implicit constants in the asymptotic $O$ notation depend only on $A$ and $\Omega$. 
We note that 
\begin{align}
\det Q(\lambda , z_1 , z_2  ) =&
Q_{11} ( Q_{22} Q_{33} - Q_{23} Q_{32})\label{term1} -  Q_{12}( Q_{21} Q_{33} -Q_{23}  Q_{31} ) \\
&+ Q_{13} ( Q_{21} Q_{32} - Q_{31} Q_{22}  ).
\end{align}
We begin by bounding the first term in \eqref{term1}. By \Cref{l:23}, we have $Q_{11} = O(1)$, so we focus on the difference $Q_{22} Q_{33} - Q_{23} Q_{32}$. 

Using \Cref{l:close} with $r = \tau$, we have 
\begin{align}\label{q22}
Q_{22} &= \frac{\iu e^{(-n/2)(z_1 - \bar z_1)^2 -n (\Im(z_1)^2 + \Im(z_1)^2)} }{2\pi \sqrt{  \Im(z_1) \Im(z_1)}} (\bar z_1 - z_1) + O(n^{-1})= \frac{1}{ \pi} + O(n^{-1}),
\end{align}
and similarly
\begin{equation}\label{q33}
Q_{33} = \frac{1}{ \pi} + O(n^{-1}).
\end{equation}
Further, 
\begin{align}\label{q23}
Q_{23} Q_{32} = 
\frac{ e^{(-n/2)(z_1 - \bar z_2)^2 + (-n/2)(z_2 - \bar z_1)^2 -2n (\Im(z_1)^2 + \Im(z_2)^2)} }{4\pi   \Im(z_1) \Im(z_2)} |\bar z_1 - z_2|^2  + O(n^{-1}).
\end{align}
Define $u_1, u_2 \in \C$ by the relations
\begin{equation}\label{eq:z_decomp_trick}
z_1 = \lambda + n^{-3/4} u_1, \qquad z_2 = \lambda + n^{-3/4} u_2.
\end{equation}
Note that 
\begin{equation}\label{z1z2}
z_1 - \bar z_2 = z_1 - \bar z_1  + n^{-3/4} (\bar u_1 - \bar u_2) = 2 \mathrm{i} \Im(z_1) + n^{-3/4} (\bar u_1 - \bar u_2)
\end{equation}
and similarly 
\begin{equation}\label{z2z1}
z_2 - \bar z_1 = 
 2 \mathrm{i} \Im(z_2) - n^{-3/4} (\bar u_1 - \bar u_2).
\end{equation}
Using \eqref{z1z2}, \eqref{z2z1}, and $\Im(z_1) - \Im(z_2) = O(n^{-3/4})$, we obtain
\begin{align}
-&\frac{n}{2}(z_1 - \bar z_2)^2 - \frac{n}{2} (z_2 - \bar z_1)^2 -2n \big(\Im(z_1)^2 + \Im(z_2)^2\big)\\
&= 
2 \iu n^{1/4}  \big(\Im(z_2) - \Im(z_1) \big)(\bar u_1 - \bar u_2)
+ O(n^{-3/2}) = O(n^{-1/2}).\label{use1}
\end{align}
Additionally, we have
\begin{equation}\label{use2}
|\bar z_1 - z_2|^2 = 4 \Im(z_1) \Im(z_2) + O(n^{-3/4}).
\end{equation}
Inserting \eqref{use1} and \eqref{use2} into \eqref{q23}, we obtain
\begin{equation}\label{q23final}
Q_{23} Q_{32} = 
\frac{ \exp\big(  O(n^{-1/2}) \big) }{\pi}  + O(n^{-3/4}) = \frac{1}{\pi} +  O(n^{-1/2}).
\end{equation}
To control the error terms in \eqref{q23final}, we used the fact that
\begin{equation}
\sup_{z,w \in \Omega : |z -w | < \tau n^{-3/4}}
\big|
U(z,w)
\big| \le C,
\end{equation}
which follows from \Cref{l:23} and \Cref{l:close}.

Combining \eqref{q22}, \eqref{q33}, and \eqref{q23final}, and recalling that we have $Q_{11} = O(1)$ from \Cref{l:23}, we conclude that 
\be\label{detterm1}
Q_{11} ( Q_{22} Q_{33} - Q_{23} Q_{32})= O(n^{-1/2}).
\ee 
Parallel reasoning (which we omit) yields 
\begin{equation}\label{detterm2}
 Q_{12}( Q_{21} Q_{33} - Q_{23} Q_{31} ) = O(n^{-1/2}), \qquad
 Q_{13} ( Q_{21} Q_{32} - Q_{31} Q_{22}  )  = O(n^{-1/2}).
\end{equation}
Inserting \eqref{detterm1} and \eqref{detterm2} into \eqref{term1}, and using \eqref{recall0}, we have 
\begin{equation}\label{this0}
\E \big[ \Xi^{(3)}(\mathcal B_M) \big]
= n^3 \int_{\Omega} \int_{D(\lambda, M_n)^2} O(n^{-1/2}) \, dz_1 \, dz_2\, d\lambda =  O(n^{-1/2}).
\end{equation}
In the last equation, we use that the area of $D(\lambda, M_n)$ is $O(n^{-3/2})$. 
Finally, \eqref{this0} implies \eqref{limE}, which finishes the proof.
\end{proof}

\subsection{Proof of Lemma~\ref{l:33}}
Given \Cref{l:bunching}, the following proof is essentially the same as the first part of the proof of \cite[Lemma 3.3]{shi2012smallest}. We give full details for completeness. 
\begin{proof}[Proof of \Cref{l:33}]
Fix a constant $\tau>0$ such that $A\subset (0, \tau)$ and set $\tau_n = n^{-3/4} \tau$ for all $n \in \N$.  For all $z\in \C$ and $r\in \R^+$, let
\[
D^+(z,r) = \{ w \in \C : |z - w | < r,  z \prec w \}.
\]

We begin by showing that for all indices $i$ such that $\lambda_i \in \Omega$ and ${i^*} \neq 0$, if 
\begin{equation}\label{indicator1}
\one_{\{\lambda_{i^*} - \lambda_i \in B_n \}}   \neq 
\one_{\{\xi^{(n)}(\lambda_i + B_n) = 1 \}},
\end{equation}
then $\xi^{(n)}(D^+(\lambda_i, \tau_n)) \ge 2$. First, suppose that the left-hand side of \eqref{indicator1} is $1$, while the right-hand side is $0$. In this case, $\xi^{(n)}(\lambda_i + B_n) \ge 2$, which implies $\xi^{(n)}(D^+(\lambda_i, \tau_n)) \ge 2$ by the definition of $B_n$. Next, suppose that the left-hand side of \eqref{indicator1} is $0$, while the right-hand side is $1$. Then  there exists some $j\neq i^*$ such that $\lambda_j \in \Omega$ and $\lambda_j \in \lambda_i + B_n \subset D^+(\lambda_i, \tau_n)$.  
Since $|\lambda_i -  \lambda_{i^*} | \le | \lambda_i - \lambda_j|$ by the definition of $i^*$, we have $\xi^{(n)}(D^+(\lambda_i, \tau_n)) \ge 2$. 
We conclude that \eqref{indicator1} implies $\xi^{(n)}(D^+(\lambda_i, \tau_n)) \ge 2$. We also note that if ${i^*} = 0$, then the term corresponding to $\lambda_i$ vanishes in both of the sums $\chi^{(n)}_{\Omega}(A )$ and $\tilde \xi^{(n)}(\Omega)$. 

From the previous paragraph, we see that 
\begin{align}\label{pointbdd}
\big|
\chi^{(n)}_{\Omega}(A )
- \tilde \xi^{(n)}(\Omega)\big|
& \le \sum_{i=1}^{n-1}
\one_{\{\xi^{(n)}(D^+(\lambda_i, \tau_n)) \ge 2\}}  \le \Xi^{(3)}(\mathcal E),
\end{align}
where $\Xi^{(3)}$ denotes the $3$-point measure defined in \eqref{xi}
and we define
\[
\mathcal E  = \{ (\lambda, z_1, z_2) : \lambda \in \Omega, (z_1, z_2) \in D^+(\lambda, \tau_n)^2\}.
\]
By \Cref{l:bunching} and the inclusion $\mathcal E \subset \mathcal B_{\tau}$, we have
\begin{equation}\label{limE}
\limsup_{n \rightarrow \infty} 
\E\big[\Xi^{(3)}(\mathcal E)\big] \le \limsup_{n \rightarrow \infty} 
\E\big[\Xi^{(3)}(\mathcal B_{\tau} )\big] = 0.
\end{equation}
By Markov's inequality and \eqref{pointbdd}, this implies \eqref{l33conclusion} and completes the proof.
\end{proof}

\subsection{Proof of Lemma~\ref{l:integration}}
We begin by recalling the following lemma from \cite[Theorem 7.8.5]{horn2012matrix}. 
\begin{lemma}\label{l:pd}
Let $M$ be an $n\times n$ positive-definite Hermitian matrix. For any $\mathcal I \subset \{1,2, \dots, n\}$, let $M_{\mathcal I}$ be the submatrix of $M$ formed by the rows and columns with indices in $\mathcal I$. We have $\det(M) \le \det (M_{\mathcal I} ) \det(M_{\mathcal I^c})$. 
\end{lemma} 

The next lemma shows that $Q^{(k)}(z_1,\dots, z_k)$ is positive-definite, or exponentially small, for all but an asymptotically vanishing set of $(z_1, \dots, z_k)$. We will use it  in conjunction with \eqref{l:pd} to bound the correlation functions $\tau_k$. It is proved in \Cref{s:auxiliary}. 

\begin{lemma}\label{l:pdset}
Fix an admissible domain $\Omega$. For all $k\in \N$, there exist constants $C_k (\Omega), c_k(\Omega) > 0$ such that the following holds. Define
\be
 \mathcal{C}_k = \{ \bs{z} \in \Omega^k :  Q^{(k)} (\bs z ) \text{ is not positive definite} \land |n^k Q_k( \bs z)| \geq C_k e^{- c_k n} \},
\ee and let $\mu$ denote Lebesgue measure on $\C^k$. Then
    \begin{equation}
        \mu(\mathcal{C}_k) < C_k e^{- c_k n}.
    \end{equation}
\end{lemma}

We require the following consequence of \Cref{l:pdset}.

\begin{lemma}\label{l:pdset2}
Fix an admissible domain $\Omega$. 
For all $k,m\in \N$ with $m > k$, there exist constants $ C_{k,m} (\Omega),  c_{k,m}(\Omega) > 0$ such that the following holds. For all $z_1, \dots, z_k \in \Omega$ and $m > k$, define 
\be
 \mathcal G_{k,m} (z_1, \dots, z_k)
= 
\big \{ 
(z_{k+1}, \dots , z_{m} ) \in \Omega^{m-k} :
(z_1, \dots, , z_m) \in
\mathcal C_{m}
\big \},
\ee
and 
\be
\mathcal C_{k,m} = 
\Big
\{ 
(z_1, \dots, z_k) \in \Omega^k :
\mu\big( \mathcal G_{k,m} (z_1, \dots, z_k) \big) > C_{k,m} e^{- c_{k,m} n}
\Big \}.
\ee
Then 
\begin{equation}
    \mu(\mathcal{C}_{k,m}) < C_{k,m} e^{- c_{k,m} n} .
\end{equation}
\end{lemma}
\begin{proof}
We consider $C_{k,m}, c_{k,m}$ as parameters that will be fixed at the end of the proof. By Fubini's theorem,
\be
\mu(\mathcal C_k) = \int_{\mathcal C_k} \, d \bs z
\ge \int_{\mathcal C_{k,m}} \int_{\mathcal G_{k,m} } \, d \bs z 
\ge \int_{\mathcal C_{k,m}} C_{k,m} e^{- c_{k,m} N} 
= \mu (\mathcal C_{k,m} ) C_{k,m} e^{- c_{k,m} N} .
\ee
We conclude from \Cref{l:pdset} that 
\be
\frac{C_k}{C_{k,m}} e^{- (c_k - c_{k,m} ) n}  \ge \mu (\mathcal C_{k,m} ).
\ee 
The conclusion follows after choosing $c_{k,m}$ such that $c_k > 2 c_{k,m}$ and $C_{k,m}$ such that $C_{k,m}^2 > C_k$. 
\end{proof}

We also recall the well-known error function asymptotic
\begin{equation}\label{erf}
\operatorname{erfc}(x) = \frac{e^{-x^2}}{\sqrt{\pi} x} 
\left( 1 + O(x^{-2}) \right),
\end{equation}
which holds as $x \rightarrow \infty$.

\begin{proof}[Proof of \Cref{l:integration}(1)]

By a standard computation using the  inclusion--exclusion principle (see \cite[(4.5)]{soshnikov2000determinantal}), we have 
\begin{multline}\label{pie}
    \tilde{\tau}_k(z_1, \ldots, z_k) = \sum^\infty_{m = 0}
        \frac{(-1)^m}{m!} \int_{z_1 + B_n} dx_1 \ldots \int_{z_k + B_n} dx_k \\
        \int_{((z_1 + B_n) \cup \ldots \cup (z_k + B_n))^m} \tau_{2k+m}(z_1, x_1, \ldots, z_k, x_k, y_1, \ldots, y_m)\,  dy_1 \ldots dy_m.
\end{multline}
We note that \cite[(4.5)]{soshnikov2000determinantal} requires that $z_i \notin z_j + B_n$ for all $i\neq j$, 
which is true for sufficiently large $n$ by the assumption that $(z_1, \dots, z_k)$ has pairwise distinct entries.

We begin by analyzing the $m=0$ term, which is
\begin{equation}
     \int_{z_1 + B_n} \dots \int_{z_k + B_n} \tau_{2k}(z_1, x_1, 
     \ldots, z_k, x_k) \, dx_1 \dots dx_k.
\end{equation}
By \eqref{taukdef} and \Cref{l:detapprox}, we have uniformly for all $z_1, x_1, \dots, z_k, x_k \in \Omega$ that
\begin{equation}\label{pear}
  \tau_{2k}(z_1, x_1, \ldots, z_k, x_k, y_1, \ldots, y_m) =   n^{2k} \det Q^{(2k)}(z_1, x_1, \ldots, z_k, x_k) + O(e^{-cn}).
\end{equation}
Observe that $Q^{(2k)}$ is $2k \times 2k$ matrix that can be written as a $k\times k$ matrix of $2 \times 2$ blocks $( \tilde Q_{ij} )_{1 \le i ,j \le k}$, where the $\tilde Q_{ij}$ have the form
\begin{align*}
    \tilde Q_{ij}=\begin{pmatrix}
        S_n\left(\sqrt{n} z_i, \sqrt{n} z_j \right)  & S_{n}\left( \sqrt{n} x_i, \sqrt{n} z_j \right) \\
        S_{n}\left( \sqrt{n} z_i, \sqrt{n} x_j\right) & S_n\left(\sqrt{n} x_i, \sqrt{n}x_j \right)
    \end{pmatrix}.
\end{align*}
For $1\le a , b \le 2$, we will will use $\tilde Q_{ij}(a ,b )$ to denote the $(a,b)$-th entry in the
$(i,j)$-th block of $Q^{(2k)}$.

We will analyze the diagonal and off-diagonal blocks separately. 
Beginning with the off-diagonal blocks where $i \neq j$, we claim that there exist constants $C, c>0$, depending only on $\Omega$ and the values $z_1, \dots, z_k$, such that 
\begin{equation}\label{offdiagonalbound}
    \max_{i \neq j} \max_{a, b\leq 2} \tilde Q_{ij}(a, b) \leq Ce^{-cn}.
\end{equation}
We will consider only the case $a = b  = 1$ in detail, since the others are nearly identical.

By \eqref{erf}, uniformly for all $z_i,z_j \in \Omega$, we have the asymptotic
expansion
\be\label{SNasymptotic2}
    S_n ( \sqrt{n} z_i , \sqrt{n} z_j) = 
    U(z_i,z_j) s_{n}(n z_i \bar z_j)
    \left( 1 + O(n^{-1}) \right).
\ee 
By \cite[Lemma 4.1]{goel2023central}, 
there exists a constant $C(\tilde d_\Omega)>0$ such that for all $z_i,z_j\in \Omega$,
\begin{equation}\label{e:original2}
s_{n}(n z_i \bar z_j)
=
1 -
\frac{1}{\sqrt{2 \pi n}}
\frac{( z_i \bar z_j  e^{1-z_i \bar z_j})^n}{1-z_i\bar z_j}
\big( 1 + R(z_i\bar z_j;n) \big), \qquad  \big|R(z_i\bar z_j; n)\big|\leq C n^{-1}.
\end{equation}
Inserting this estimate into \eqref{SNasymptotic2}, we find 
\begin{align}
    |S_n(\sqrt{n} z_i, \sqrt{n} z_j)| \leq 
        C e^{-(n/2) \Re (z_i - \bar {z}_j)^2}  e^{-n (\Im(z_i)^2 + \Im (z_j)^2)} \\
    \times 
    \left(
        1 + C \left| e^{-2 (1 - z_i \bar z_j)} 
        (z_i \bar z_j)^n \right| 
    \right). 
\end{align}
Next, we note that there exists a constant $c(z_1, \dots, z_n) > 0$ such that  
\begin{equation}
    e^{-(n/2) \Re (z_i - \bar z_j)^2} e^{-n (\Im(z_i)^2 + \Im(z_j)^2)} = e^{-(n/2) |z_i - z_j|^2} \le e^{- cn },
\end{equation}
and 

\begin{align}
    \left|e^{-(n/2) \Re(\lambda - \bar z_j)^2} e^{-n (\Im(z_i)^2 + \Im(z_j)^2)}\right| 
    &\leq e^{(n/2) (- |z_i - \bar z_j|^2 + 2 + 2 \Re(z_i \bar z_j) - 2 \ln |z_i \bar z_j|} \\
    &= e^{-(n/2) (|z_i|^2 + |z_j|^2 - 2 - \ln |z_i|^2 - \ln |z_j|^2 ))} \\
    & \leq e^{-cn},
\end{align}
where the last inequality follows from $|z|^2 - 1 - \ln |z|^2 \geq c >  0$ for $z \in \Omega$ (which holds by the convexity of $x\mapsto \ln x$).
This completes the proof of \eqref{offdiagonalbound}.

We now consider the diagonal blocks, corresponding to the case $i = j$. By \Cref{l:close}, these take the form
\begin{equation}\label{eq:sec4_diag_block_Sn}
    \begin{pmatrix}
        S_n(z_i, z_i) & S_n(z_i, x_i) \\
        S_n(x_i, z_i) & S_n(x_i, x_i)
    \end{pmatrix}
    =   \begin{pmatrix}
        U(z_i, z_i) & U(z_i, x_i) \\
        U(x_i, z_i) & U(x_i, x_i)
    \end{pmatrix} + O(n^{-1}),
\end{equation}
where the additive error term denotes a matrix such that each entry is $O(n^{-1})$.  Since $|U(z,w)| \le C$ uniformly for $z,w \in \Omega$, by \Cref{l:close} and \Cref{l:23}, we have 
\begin{equation}\label{thedet}
   \det  \begin{pmatrix}
        S_n(z_i, z_i) & S_n(z_i, x_i) \\
        S_n(x_i, z_i) & S_n(x_i, x_i)
    \end{pmatrix}
    =  \det  \begin{pmatrix}
        U(z_i, z_i) & U(z_i, x_i) \\
        U(x_i, z_i) & U(x_i, x_i)
    \end{pmatrix} + O(n^{-1}).
\end{equation}
We now compute the determinant on the right-hand side of \eqref{thedet}.
Beginning with the term coming from the diagonal of this block, by direct substitution, we observe that
\begin{equation}\label{det1}
    U(z_i, z_i) = U(x_i, x_i) = \frac{1}{\pi}.
\end{equation}
For the term coming from the cross-diagonal,
observe first that
\begin{multline}\label{eq:sec4_cross_diag_product_long}
    U(z_i, x_i) U(x_i, z_i) = \frac{-1}{4\pi^2 \Im(x_i) \Im(z_i)} \cdot (\bar{x}_i - z_i)(\bar{z}_i - x_i) \\ \times \operatorname{exp}\left( -\frac{n}{2} \left[(z_i - \bar{x}_i)^2 + (x_i - \bar{z}_i)^2\right] - 2n \left[\Im(z_i)^2 + \Im(x_i)^2\right] \right)
\end{multline}
Following the same strategy as in \eqref{eq:z_decomp_trick}, since $x_i \in z_i + B_n$, we will define $u_i \in \C$ by the equality
\begin{equation}
    x_i = z_i + n^{-3/4} u_i.
\end{equation}
We now return to \eqref{eq:sec4_cross_diag_product_long} and analyze each factor separately. First, we have
\begin{align}
\frac{-1}{4\pi^2 \Im(x_i) \Im(z_i)} & = \frac{-1 }{4\pi^2 \Im(z_i) \left( \Im(z_i) + n^{-3/4} \Im(u_i) \right)}\notag \\
 & = \frac{ - 1}{ 4 \pi^2 \Im(z_i)^2} + O(n^{-3/4}).
\end{align}
Secondly, we have
\begin{align}
    (\bar{x}_i - z_i)(\bar{z}_i - x_i) &= (-2i\Im(z_i) + n^{-3/4} \bar{u}_i) (-2i \Im(z_i) - n^{-3/4}u_i)\notag \\
    &= -4\Im(z_i)^2 + O(n^{-3/4}).
\end{align}
Next, in the exponent, we have
\begin{align}
    -\frac{n}{2}&\left((z_i - \bar{x}_i)^2 + (x_i - \bar{z}_i)^2\right)\notag \\ &= -\frac{n}{2} \left((2i \Im(z_i) - n^{-3/4}\bar{u}_i)^2 + (2i\Im(z_i) + n^{-3/4} u_i)^2\right)\notag \\
        &= -\frac{n}{2} \left(-8\Im(z_i)^2 - 8n^{-3/4} \Im(z_i) \Im(u_i) + n^{-3/2} (u_i^2 + \bar{u}_i^2)\right)\notag \\
        &= 4n\Im(z_i)^2  + 4n^{1/4} \Im(z_i) \Im(u_i) - n^{-1/2} (\Re(u_i)^2 - \Im(u_i)^2),\label{eq:SEC4_first_exponent_term}
\end{align}
and, for the other term in the exponent,
\begin{align}
    - 2n \left(\Im(z_i)^2 + \Im(x_i)^2\right) &= -2n \left( 2\Im(z_i)^2 + 2n^{-3/4} \Im(z_i)\Im(u_i) + n^{-3/2} \Im(u_i)^2 \right)\notag  \\
    &= -4n\Im(z_i)^2 - 4n^{-1/4} \Im(z_i) \Im(u_i)^2 - 2n^{-1/2} \Im(u_i)^2.\label{eq:SEC4_second_exponent_term}
\end{align}
Summing \eqref{eq:SEC4_first_exponent_term} and \eqref{eq:SEC4_second_exponent_term}, we see that
the exponent is
\begin{align}
     - n^{-1/2}(\Re(u_i)^2 - \Im(u_i)^2) -2n^{-1/2} \Im(u_i)^2 &=  - n^{-1/2}| u_i|^2 \\   &= - n| z_i - x_i|^2,
\end{align}
noting that $u_i = n^{3/4}(x_i - z_i)$. So, \eqref{eq:sec4_cross_diag_product_long} simplifies to
\begin{equation}\label{det2}
    U(z_i, x_i) U(x_i, z_i) =  \frac{1}{\pi^2} \exp\left( - n |z_i - x_i|^2\right) + O(n^{-3/4}).
\end{equation}
Using \eqref{det1} and \eqref{det2}, we compute the determinant on the right-hand side of \eqref{thedet} and find 
\begin{align}\label{turkey}
    \det \begin{pmatrix}
        U(z_i, z_i) & U(z_i, x_i) \\
        U(x_i, z_i) & U(x_i, x_i)
    \end{pmatrix} + O(n^{-1}) &=\pi^{-2} \left(1 - \exp(-n|z_i - x_i|^2)\right) + O(n^{-3/4}).
\end{align}

Returning to \eqref{pear}, and using \eqref{offdiagonalbound}, we see that the only non-negligible contributions to the determinant of $Q^{(2k)}$ are from the
$2\times 2$ blocks along the diagonal. Up to an additive error of $O(n^{-3/4})$, this is
\begin{equation}\label{basil}
    \prod^k_{i=1} \left( n^2 \pi^{-2} \int_{z_i + B_n} (1 - \exp(-n|z_i - x_i|^2) \, dx_i \right).
\end{equation}
It suffices to compute each integral from this product individually. Making the change of variable $u = n^{3/4}(x_i - z_i)$, we compute
\begin{align}
n^2 \int_{z_i + B_n} \big(1 - \exp(-n|z_i - x_i|^2\big) \, dx_i 
 &= n^{1/2} \int_{B} \big (1 - \exp(-n^{-1/2} |u|^2 \big) \, du \\
 &= \int_B |u|^2 \,du + O(n^{-1/2}).
\end{align}
We conclude that \eqref{basil} equals
\be
\left(\frac{1}{\pi^2} \int_B |u|^2 \,du \right)^k + O(n^{-1/2}).
\ee 
For future reference, we also note that \eqref{turkey} implies that 
\be\label{turkey2}
\sup_{z\in \Omega}  \left| \int_{z + B_n} n^2 \det Q^{(2)} (z, x )  \right| \le C.
\ee

It remains to show that the contributions from the terms  in \eqref{pie} with $m\ge 1$ are negligible. We begin with the indices $m \ge k + 5$. We can bound such a term by 
\begin{multline}\label{largem}
   \left| \frac{(-1)^m}{m!} \int_{z_1 + B_n}  \ldots \int_{z_k + B_n}  \int_{((z_1 + B_n) \cup \ldots \cup (z_k + B_n))^m} \tau_{2k+m} (z_1, x_1, \ldots, z_k, x_k, \bs y) \,  d \bs{y}  \, d \bs{x}\right| \\ 
   \le \frac{1}{m!} \cdot C^m k^m n^{-3m/2} \cdot C^k n^{-3k/2}\cdot C^{2k+m} n^{2k+m}  = \frac{C^{3k + 2m}}{m!} k^m n^{(k - m)/2} ,
\end{multline}
where we bound the area of integration of $\bs y$ by $(C k n^{-3/2})^m $, we bound the area of integration of $\bs x$ by $C^k n^{-3k/2}$, and we use the bound $\tau_{2k+m} \le C^{2k+m} n^{2k+m}$ coming from \eqref{pfcorr}, \Cref{r:23}, and \Cref{l:23}. (Observe that the $C$ in the last inequality is independent of $k$ and $m$.) We conclude that the sum of all terms in \eqref{pie} with $m > k+5$ is  $O(n^{-1/2})$. 

Next, we consider the terms with $m < k + 5$. Fix $m < k+5 $ and set $\ell = 2k+m$. We use the notation $C_{k,\ell}^{(n)}$ to make the dependence of the set $C_{k,\ell}$ on $n$ explicit. Using \Cref{l:pdset2}, we compute 
\be
\sum_{n=1}^\infty \mu\big( \mathcal C_{k,\ell}^{(n)} \big)
\le C_{k,\ell} \sum_{n=1}^\infty e^{- c_{k, \ell} n} < \infty.
\ee
Therefore, by the Borel--Cantelli lemma, for every $\bs z  \in \Omega^k$ (except a set of measure zero), there exists $n_0(\bs z)$ such that $\bs z \in  C_{k,\ell}^c$ for every $n \ge n_0$. Since we are proving an asymptotic statement, we suppose $n \ge n_0$ for the rest of the proof. Then, since $\mathcal G_{k,\ell}(z_1, \dots, z_{k})$ is exponentially small for such $\bs z \in \Omega^k$ (by the definition of $\mathcal C_{k, \ell}$), it remains to bound the integral over $\mathcal G_{k,\ell}^c(z_1, \dots, z_{k})$ in the $m$-th term of \eqref{pie}. 

Define  
\be
\mathcal{A}_\ell = \{ (z_{k+1}, \dots, z_{\ell} \in \Omega^{\ell - k} : Q^{(\ell)}({\bf{z}}) \text{ is positive definite} \},
\ee
and set $\mathcal B_{\ell} = \mathcal G_{k,\ell}^c(z_1, \dots, z_{k}) \setminus \mathcal A_\ell$. By the definition of $\mathcal G_{k,\ell}$, for  $(z_{k+1}, \dots, z_{\ell} ) \in \mathcal B_{\ell}$ we have 
\be\label{Bbound}
n^\ell \big  | \det Q^{(\ell)} (z_1, \dots,z_{\ell})
\big|  \le C e^{-cn}
\ee 
for some constants $C,c>0$ depending only on $k$ and $\Omega$. In particular, the constants are uniform in $\ell$ (and hence $m$) by the assumed upper bound on $m$. Using \eqref{Bbound}, \eqref{pear}, and that $\tau_{2k+m}$ is symmetric in its arguments, we can bound the portion of the integral in the $m$-th term over the region $\mathcal B_\ell$ by
\be
\int_{\mathcal B_\ell} \det Q^{(2k+m)}(z_1, \dots, z_k,  x_1, \dots, x_k, y_1, \ldots, y_m)\, d\bs x\,  d\bs y \le C e^{-cn},
\ee 
after increasing the value of $C$.

Finally, we consider the integral over $\mathcal 
A_\ell$ where $Q^{(2k+m)}$ is positive-definite.
We observe that the integrand here can be estimated by
\begin{equation}\label{eq:sec4_3p1_integrand}
    \one_{\mathcal A_{\ell}} \det Q^{(2k+m)}(z_1, \dots, z_k,  \bs x , \bs y) 
    \le 
    \one_{\mathcal A_{\ell}}  \det Q^{(2k)}(z_1, \dots, z_k, \bs x)  \left( \prod_{i=1}^m \det Q^{(1)} (y_i) \right) 
\end{equation}
by iteratively applying \Cref{l:pd} since 
the matrices $M_{\mathcal I}$ and $M_{\mathcal I^c}$ 
as defined in \Cref{l:pd} are positive-definite whenever 
$M$ is (since all principal minors of a 
positive-definite matrix are positive-definite).
By \Cref{l:23}, we have 
\begin{equation}\label{tiger}
    \left( \prod_{i=1}^m \det Q^{(1)} (y_i) \right) \le C^m.
\end{equation}
Bounding \eqref{eq:sec4_3p1_integrand} using the estimates \eqref{eq:sec4_3p1_integrand} and \eqref{tiger}, we find
\begin{multline}
\left| 
\int_{((z_1 + B_n) \cup \ldots \cup (z_k + B_n))^m} 
\int_{z_1 + B_n}  \ldots \int_{z_k + B_n}  \one_{\mathcal A_{\ell}} n^{2k+m} \det Q^{(2k+m)}(z_1, \dots, z_k,  \bs x , \bs y)\, d\bs x\,  d\bs y \right| \\
\le Cn^{m} \cdot n^{-3 m /2} \int_{z_1 + B_n}  \ldots \int_{z_k + B_n} \one_{\mathcal A_{\ell}}  n^{2k} \det Q^{(2k)}(z_1, \dots, z_k, \bs x) \\
\le Cn^{m}  \cdot n^{-3m/2} = C n^{- m/2},\label{aspen}
\end{multline}
where the constant $C$ depends on $m$ and $\Omega$, and changes at each appearance. The bound in \eqref{aspen} is $o(1)$, which completes the proof.
\eeb
\end{proof}

\begin{proof}[Proof of \Cref{l:integration}(2)]
We retain the notations from the previous proof. Our goal is to prove a uniform upper bound on \eqref{pie}. The same argument that gave \eqref{largem} shows that the that the sum of all terms in \eqref{pie} with $m \ge k+5$ is  $O(n^{-1/2})$. 
For the terms with $m < k + 5$, the previous Borel--Cantelli argument no long suffices, as it is not uniform in $n$. Therefore, we reason as follows. Define $\mathcal X_k = \mathcal X_k^{(n)} \subset \Omega^k$ by 
\be
\mathcal X_k = \mathcal C^c_k \cap \left( \bigcap_{m=1}^{k+4} \mathcal C_{k,m}^c \right).
\ee
For $\bs z \in \mathcal X_k$,  the same argument that gave \eqref{aspen} shows the the terms with $0 \le m \le k+4$ in \eqref{pie} are uniformly bounded by a constant $C>0$ (which depends on $k$ and $\Omega$). By \eqref{l:pdset} and \eqref{l:pdset2}, there exist constants $C, c > 0$ (depending on $k$) such that 
\be\label{Xmeas0}
\P( \mathcal X^c _k)  \le C e^{-cn}.
\ee
Then with the choice $\mathcal W_n = \mathcal X^{(n)}_k$, we have established \eqref{W1} and \eqref{W2}. For \eqref{W3}, we note that \eqref{pfcorr} and \eqref{l:23} together show that that there exists a constant $\tilde C>0$ such that for every $k \in \N$, we have 
\be\label{globalbound}
\sup_{z_1, \dots, z_k \in \Omega} \tau_{k} (z_1, \dots, z_k) \le \tilde C^k n^k.
\ee 
We emphasize that $\tilde C$ does not depend on $k$. Using \eqref{globalbound} in the terms with $0 \le m \le k+4$ in \eqref{pie} shows that their sum is bounded by
\be
C n^{3k + 4} \le C n^{ 8 k}
\ee 
for some $C(k) > 0$, where we use $k\ge 1$. This shows that \eqref{W3} holds and completes the proof.
\end{proof}

\begin{proof}[Proof of \Cref{l:integration}(3)]

Consider an arbitrary point $(z_1, \dots, z_k) \in \overline{\Psi}_k$. 
We begin by defining an equivalence relation $\sim$ on the set $\{1, \dots, k\}$ as follows. This equivalence relation is determined by two conditions. First, we have $i \sim j$ if $z_i - z_j \in B_n$ or $z_j - z_i \in B_n$. Second, we also have ${i} \sim j$ if there exists a sequence $i_1, \dots, i_s$ of indices such that $i_1 = i$, $i_s = j$, and $z_{i_m} \sim z_{m+1}$ for all $1\le m < s$. 
Let $p$ denote the number of equivalence classes under $\sim$. For each $r \in \N$ with $1 \le r \le p$, let $i_r$ denote the index in the $r$-th equivalence class such that $z_{i_r}$ is the maximal element in this equivalence class under $\prec$. (The set of $\bs z \in \overline{\Psi}_k$ where there is not a unique maximal element in each class has measure zero, and therefore can be neglected.)

By the definition of $\tilde \xi^{(n)}$ and \Cref{l:detapprox}, there exist $C(\Omega), c(\Omega)>0$ such that
\begin{align}
\tilde \tau_k &(z_1, \dots, z_k) \\
&\le 
\int_{z_{i_1}+B_n} 
\dots
\int_{z_{i_p}+B_n} 
\tau_{k+p} ( z_1, \dots, z_k, x_1, \dots, x_p)\, d x_1 \dots d x_p \\
& \le n^{k+p}
\int_{z_{i_1}+B_n} 
\dots
\int_{z_{i_p}+B_n} 
\det Q^{(k+p)} ( z_1, \dots, z_k, x_1, \dots, x_p)\, d x_1 \dots d x_p + C e ^{-cn}.\label{orange}
\end{align}
We define $\mathcal Y_k \subset \Omega^k$ by 
\be
\mathcal Y_k = \mathcal C_k^c \cap \left( \bigcap_{m=1}^{2k} \mathcal C_{k,m}^c\right).
\ee 
From \eqref{l:pdset} and \eqref{l:pdset2}, there exist constants $C, c > 0$ such that 
\be\label{Xmeas}
\P( \mathcal Y_k) > 1 - C e^{-cn}.
\ee
Note that by \Cref{l:23} and \Cref{l:detapprox}, there exists a constant $C(\Omega)>0$ such that 
\be\label{naive}
\big| \det Q^{(k+p)} ( z_1, \dots, z_k, x_1, \dots, x_p)\big|
\le C^{k+p} \le C^{2k}, 
\ee 
since $p < k$ by the definitions of $p$ and $\overline{\Psi}_k$.
Then \eqref{Xmeas} and \eqref{naive} together imply that 
\be\label{Xcnegligible}
\lim_{n\rightarrow \infty}
\int_{\mathcal Y^c_k} \tilde \tau_k (z_1, \dots, z_k)
=0.
\ee 
It remains to consider the analogous integral over the region $\mathcal Y_k \cap \overline{\Psi}_k$.

Returning to \eqref{orange}, we let $\mathcal A_{k,p} \subset \mathcal Y_k$ denote the set of points $(z_1, \dots, z_{k+p})$ such that $Q^{(k+p)}(z_1,\dots, z_{k+p})$ is positive-definite, and set $\mathcal B_{k,p} = \mathcal Y_k \setminus \mathcal A_{k,p}$. 
On $\mathcal B_{k,p}$, we have  $|\tau_{k+p}(\bs z)| \le C e^{-cn}$ uniformly for all $\bs z$ and all $ p <k$, so the contribution from $\mathcal B_{k,p}$ is negligible by the same argument that gave \eqref{Xcnegligible}. 

On $\mathcal A_{k,p}$, we have after using \Cref{l:pd} multiple times that 
\begin{multline}
n^{k+p} \int_{z_{i_1}+B_n} 
\dots
\int_{z_{i_p}+B_n} \one_{\mathcal A_{k,p}} \det 
Q^{(k+p)} ( z_1, \dots, z_k, x_1, \dots, x_p)\, d x_1 \dots d x_p\\
\le 
\prod_{j \neq i_1, \dots i_p} 
n \det Q^{(1)} (z_j)
\cdot 
\prod_{j=1}^p \int_{z_{i_j} + B_n} n^2 \det Q^{(2)} (z_{i_j}, x_j )\, d x_j.\label{blueberry}
\end{multline}
It follows from the estimate \eqref{turkey2} above that 
\be\label{blueberry1}
\sup_{z_1, \dots, z_p \in \Omega}
\left| \prod_{j=1}^p \int_{z_j + B_n} n^2 \det Q^{(2)} (z_j, x_j )\, d x_j \right| \le C^p.
\ee
We also have, using \eqref{pfcorr}, \Cref{r:23}, and \Cref{l:24}, that
\be\label{blueberry2}
\sup_{z_1, \dots, z_{k-p}}
\left|
\prod_{j =1}^{k-p}
n \det Q^{(1)} (z_j)
\right| \le (C n)^{k-p}.
\ee
For a given $p$, the measure of the region of integration in \eqref{blueberry} is bounded by $ C_p n^{-3(k-p)/2}$.  Then inserting \eqref{blueberry1} and \eqref{blueberry2} into the right-hand side of \eqref{blueberry}, we find that the integral is bounded is $O(n^{-k/2 + p/2})$. Since $k > p$ by the definition of $\overline{\Psi}_k$, this finishes the proof.
\end{proof}

\section{Proofs of Auxiliary Lemmas}\label{s:auxiliary}

The following proof is similar to the proof of \cite[Lemma 3.3]{goel2023central}. We give it here for the reader's convenience. 
\begin{proof}[Proof of \Cref{l:detapprox}]
We recall from \eqref{pfcorr} that  $\rho_k(\sqrt{n} z_1, \dots , \sqrt{n} z_k)$ can be written explicitly as a Pfaffian.  Using \Cref{l:23}, all terms in this expansion involving a factor $D_n$ or $I_n$ are exponentially small. We conclude that  
\begin{equation*}
\sup_{z_1,\dots, z_k \in \Omega}
\big|
n^k \rho_k(\sqrt{n} z_1, \dots, \sqrt{n} z_k)
-
n^k \pf\big(\tilde K(\sqrt{n} z_i, \sqrt{n} z_j)\big)_{1 \leq i, j \leq k},
\big|
\le c^{-1} e^{-cn},
\end{equation*}
where $(\tilde K(z_i, z_j))_{1 \leq i, j \leq k}$ is a $2k \times 2k$ matrix composed of the $2\times 2$ blocks
\begin{align*}
    \tilde K(z_i, z_j)=\begin{pmatrix}
0  & S_{n}\left(z_i, z_j\right) \\
-S_{n}\left(z_j, z_i\right) & 0 
\end{pmatrix}.
\end{align*}
The conclusion follows after noting that \Cref{l:24} implies 
\begin{equation*}
\pf\big(\tilde K(\sqrt{n} z_i, \sqrt{n} z_j)\big)_{1 \leq i, j \leq k}
= \det Q^{(k)}(z_1, \dots, z_k)
.
\end{equation*}
\end{proof}

\begin{proof}[Proof of \Cref{l:close}]
By \eqref{erf}, for all $z,w \in \Omega$, we have the asymptotic expansion
\begin{equation}\label{SNasymptotic}
    S_n ( \sqrt{n} z , \sqrt{n} w) = 
    U(z,w) s_{n}(n z \bar w)
    \left( 1 + O(n^{-1}) \right).
\end{equation}
By \cite[Lemma 4.1]{goel2023central}, 
there exists a constant $C(\tilde d_\Omega)>0$ such that for all $z,w\in \Omega$,
\begin{equation}\label{e:original}
s_{n}(n z \bar w)
=
1 -
\frac{1}{\sqrt{2 \pi n}}
\frac{( z \bar w  e^{1-z \bar w})^n}{1-z\bar w}
\big( 1 + R(z\bar w;n) \big), \qquad  \big|R(z\bar w; n)\big|\leq C n^{-1}.
\end{equation}
Considering the case $z=w$, the exponential factor becomes
\begin{equation*}
|z|^2 e^{1-|z|^2} = \exp\big( 1 - |z|^2 + \log(|z|^2)  \big).
\end{equation*}
There exists a constant $\eps(\Omega) > 0$ such that for all $z \in \Omega$, we have $\eps < |z|^2 < 1 - \eps$. Then, using concavity of the logarithm, we find that there exists $\delta(\Omega) > 0$ such that for all $z \in \Omega$,
\begin{equation*}
 1 - |z|^2 + \log(|z|^2) < - \delta.
\end{equation*}
By continuity, it follows that there exists $n_0(r, \Omega) > 0$ such that for all $z,w\in \Omega$ such that $|z-w| < r n^{-3/4}$, and all $n > n_0$, 
\begin{equation}\label{e:prev}
\Re\big( 1 - z\bar w - \log (z \bar w) \big) < - \delta/2 , \qquad 
| z \bar w  e^{1-z \bar w} |^n \le e^{ -\delta n /2 }.
\end{equation}
Inserting \eqref{e:prev} into \eqref{e:original} shows that 
\begin{equation}\label{e:penultimate}
s_{n}(n z \bar w) = 1 + O(n^{-1}),
\end{equation}
after noting that $(1-z\bar w)^{-1}$ is uniformly bounded above for all $z,w \in \Omega$. Then inserting \eqref{e:penultimate} into \eqref{SNasymptotic} completes the proof (after recalling that the left-hand side of \eqref{SNasymptotic} is $O(1)$, by \Cref{l:23}).
\end{proof}
\begin{proof}[Proof of \Cref{p:poisson}]
Note that $X_n \rightarrow \chi(J)$ in distribution is implied by the statement that $\E[X^r] \rightarrow \E[X^r]$ for all $r \in \N$, by \cite[Theorem 30.1]{billingsley2017probability} and \cite[Theorem 30.2]{billingsley2017probability}.\footnote{For the application of \cite[Theorem 30.1]{billingsley2017probability}, we recall that the moment generating function of a Poisson distribution with rate $\lambda$ is given by $\exp(\lambda( e^t  -1 ))$, whose power series converges for all $t \in \R$.} Convergence of all  moments $\E[\chi^{(n)}(J)^r]$ is implied by the convergence of the factorial moments on the left-hand side of \eqref{factorialmoments}, since the factorial moments of order at most $r$ determine the moment $\E[\chi^{(n)}(J)^r]$, and this completes the proof after recalling the standard fact that the $r$-th factorial moment of a Poisson random variable with rate $\lambda$ is $\lambda^r$ (which follows from a straightforward computation using the probability density function). 
\end{proof}
\begin{proof}[Proof of \Cref{p:bab}]
This follows directly from \cite[Theorem 4.15]{kallenberg2017random}; see the comment immediately before its statement. Note that our hypothesis on $\mu$ implies that $\chi$ is simple, as required by theorem. We are also using that the collection of bounded Borel subset of $\R$ forms a \emph{dissecting ring} according to the definition of \cite[p. 24]{kallenberg2017random}.
\end{proof}

For the proof of \Cref{l:pdset}, we need the following lemma.

\begin{lemma}\label{l:integrate_out_lemma}
    Fix an admissible domain $\Omega$ and $k,m, n \in \N$ such that $m < k < n$. Then
    \begin{equation}\label{51conclude}
        \tau_m (z_1, \ldots, z_m) =  \frac{(n-k)!}{(n-m)!}  
        \int_{\Omega^{k-m}} \tau_k (z_1, \ldots, z_k) \, dz_{m+1} \dots dz_k.
    \end{equation}
\end{lemma}
\begin{proof}
Fix a bounded Borel function $f: \Omega^m \to \R$, and define $\phi : \Omega^k \to \R$ by $\phi(z_1, \ldots, z_k) = f(z_1, \ldots, z_m)$. We know that $f$ satisfies the equality in \eqref{corrdef} for $\tau_m$ and that $\phi$ satisfies
    the analogous equality for $\tau_k$. Analyzing the term on the right-hand side of \eqref{corrdef}, we have that
    \begin{align}
        \sum_{(i_1, \ldots, i_k) \in \mathcal{I}_k} \phi(w_{i_1}, \ldots, w_{i_k}) &= \sum_{(i_1, \ldots, i_k) \in \mathcal{I}_k} f(w_{i_1}, \ldots, w_{i_m})\\
        &= \sum_{(i_1, \ldots, i_m) \in \mathcal{I}_m}  \frac{(n-m)!}{(n-k)!} f(w_{i_1}, \ldots, w_{i_m}).
    \end{align}
   Using \eqref{corrdef}, this implies that
    \begin{multline}\label{rhs11}
        \int_{\Omega^k} \phi(z_1, \ldots, z_k) \tau_k(z_1, \ldots, z_k)\, dz_1\ldots dz_k = \\ \frac{(n-m)!}{(n-k)!} \int_{\Omega^m} f(z_1, \ldots, z_m) \tau_m (z_1, \ldots, z_m) \, dz_1 \ldots dz_m.
    \end{multline}
    Additionally, by Fubini's theorem and the definition of $\phi$,
    \begin{multline}
        \int_{\Omega^k} \phi(z_1, \ldots, z_k) \tau_k(z_1, \ldots, z_k) \, dz_1 \ldots dz_k = \\ 
         \int_{\Omega^m} f(z_1, \ldots, z_m) \left( \int_{\Omega^{k-m}} \tau_k (z_1, \ldots, z_k) \, dz_{m+1} \ldots dz_k\right) dz_1 \ldots dz_m.\label{rhs21}
    \end{multline}
Since $f$ was arbitrary, we conclude from comparing the right-hand sides of \eqref{rhs11} and \eqref{rhs21} that \eqref{51conclude} holds.
\end{proof}

\begin{proof}[Proof of \Cref{l:pdset}]
Define
\be
\mathcal{A} = \{ {\bf{z}} \in \Omega^k : Q^{(k)}({\bf{z}}) \text{ is positive definite} \}
\ee
and let $\mathcal{D} = \Omega^k \setminus \mathcal{A}$. We let $M_i$ denote the $i$-th
    leading principal minor of $Q^{(k)}$, obtained by removing the last $k-i$ rows and
    columns. For $1 \leq j \leq k$, we define
    \begin{equation}
        \mathcal{D}_j = \{ \bs z \in \Omega^k \mid \det M_j(\bs z) \le 0 \land  \forall \ell < j \, \det M_\ell(\bs z) >  0 \}.
    \end{equation}
    In other words, $\mathcal D_j$ is the set of $\bs z$ where the smallest principal minor with a non-positive determinant is the $j$-th one. 
    By Sylvester's criterion (\cite[Theorem 7.2.5(b)]{horn2012matrix}), we may
    write $\mathcal{D}$ as the  disjoint union of $\mathcal{D}_j$ for
    $j = 1, \ldots, k$:
    \begin{equation}
        \mathcal{D} = \mathcal{D}_{k} \sqcup \mathcal{D}_{k-1} \sqcup \ldots \sqcup \mathcal{D}_1.
    \end{equation}

    Note that membership in $\mathcal D_m$ is determined by the first $m$ coordinates, meaning that if $(z_1, \dots, z_k) \in \mathcal D_m$, then  for all $(w_{m+1}, \dots, w_k) \in \Omega^{k-m}$, we have 
    \be (z_1, \dots, z_m, w_{m+1}, \dots , w_k )  \in \mathcal D_m.
    \ee 
With this in mind, we define 
\be
\mathcal D'_m = \{ (z_1, \dots, z_m )\in \Omega^m : 
\exists (z_{m+1}, \dots, z_k) \in \Omega^{k-m} \text{ with } (z_1, \dots, z_k) \in \mathcal D_m 
\}.
\ee

    Fix $m \le k$. By \Cref{l:detapprox}, there exists $c(k, d_\Omega)>0$ such that for all $(z_1, \ldots, z_m) \in \mathcal{D}'_m$,  we have
    \begin{equation}\label{taunegative}
        \tau_m (z_1, \ldots, z_m) \leq c^{-1} e^{-cn},
    \end{equation}
    because $\tau_m \ge 0$ and $\det Q^{(m)}(z_1,\dots, z_m) \leq 0$ on $\mathcal D_m$ (since $Q^{(m)}$ is the $m$-th principal minor of $Q^{(k)}$). 
    
    By \Cref{l:integrate_out_lemma}, it follows that there exists a constant $\hat c  > 0$ such that $c > \hat c$, and for all $(z_1, \ldots, z_m) \in \mathcal{D}'_m$, 
    \begin{align}
        \int_{\Omega^{k-m}} \tau_k (z_1, \dots, z_m, w_{m+1}, \dots , w_k )  \, dw_{m+1} \ldots dw_k &= \frac{(n-m)!}{(n-k)!} \tau_m(z_1, \ldots, z_m) ,
         \\ &\leq n^k \tau_m(z_1, \ldots, z_m)\\
          & \le \hat c^{-1} e^{-\hat c n}. \label{purple}
    \end{align}
    
    Now, if we define
    \begin{equation}
        D(z_1, \ldots, z_m) = \{(z_{m+1}, \ldots, z_k) \in \Omega^{k-m}:  \tau_k ( z_1, \ldots, z_k) \geq \hat c^{-1} e^{-\hat c  n /2 }\},
    \end{equation}
    then by \eqref{purple} and since $\tau_k \ge 0$, we have for all $(z_1,\dots, z_m) \in \Omega^m$ that 
    \begin{equation*}
        \int_{D(z_1, \ldots, z_m)}\tau_k (z_1, \dots, z_m, w_{m+1}, \dots , w_k )  \, dw_{m+1} \ldots dw_k  \leq \hat c^{-1} e^{- \hat c n},
    \end{equation*}
       which implies by the definition of $D(z_1, \ldots, z_m)$ that 
    \begin{equation*}
        \mu(D(z_1, \ldots, z_m)) \leq e^{- \hat c n / 2}.
    \end{equation*}
    Define $\mathcal C_k^{(m)} \subset \mathcal D_k$ by 
    \be 
    \mathcal C_k^{(m)}  = \{ (z_1, \dots , z_k) \in \mathcal D_k :
    \tau_k (z_1, \ldots, z_k) \ge \hat c^{-1} e^{- \hat c n/2}\}.
    \ee 
    Then 
    \be
    \mu\big(\mathcal C_k^{(m)} \big) =  \int_{(z_1, \dots, z_m) \in \mathcal D_m'} \int_{D(z_1, \ldots, z_m)} \, d \bs z \le | \Omega |^m  \hat c^{-1} e^{-\hat c n} \le  2^m \hat c^{-1} e^{-\hat c n} .
    \ee
    Since $m \le k$, we deduce the existence of a constant $\tilde c > 0$ such that $\hat c/2  > \tilde c$ and 
    \be\label{777}
    \sum_{m=1}^k
    \mu \big ( \mathcal C_k^{(m)} \big)  \le \tilde c^{-1} e^{- \tilde c n}. 
    \ee
    Then the desired conclusion holds after setting $c_k = \tilde c$ and $C_k = \tilde c^{-1}$, since with these choices, we have 
    \be
    \mathcal C_k \subset \bigcup_{m=1}^k  \mathcal C_k^{(m)},
    \ee
    and hence 
    \be
    \mu( \mathcal C_k) 
    \le \mu\left( \bigcup_{m=1}^k  \mathcal C_k^{(m)} \right) \le \sum_{m=1}^k
    \mu \big ( \mathcal C_k^{(m)} \big) \le \tilde c^{-1} e^{- \tilde c n}
    \ee
by \eqref{777} and a union bound.
\end{proof}


\begin{thebibliography}{10}

\bibitem{ameur2018repulsion}
Yacin Ameur.
\newblock Repulsion in low temperature $\beta$-ensembles.
\newblock {\em Communications in Mathematical Physics}, 359(3):1079--1089,
  2018.

\bibitem{ameur2023planar}
Yacin Ameur and Jos{\'e}~Luis Romero.
\newblock The planar low temperature {C}oulomb gas: separation and
  equidistribution.
\newblock {\em Revista Matem{\'a}tica Iberoamericana}, 39(2):611--648, 2023.

\bibitem{ben2013extreme}
Gerard Ben~Arous and Paul Bourgade.
\newblock Extreme gaps between eigenvalues of random matrices.
\newblock {\em The Annals of Probability}, 41(4):2648--2681, 2013.

\bibitem{billingsley2017probability}
Patrick Billingsley.
\newblock {\em Probability and {M}easure}.
\newblock John Wiley \& Sons, third edition, 2008.

\bibitem{bleher2006zeros}
Pavel Bleher and Robert Mallison~Jr.
\newblock Zeros of sections of exponential sums.
\newblock {\em International Mathematics Research Notices}, 2006:38937, 2006.

\bibitem{bohigas1984characterization}
Oriol Bohigas, Marie-Joya Giannoni, and Charles Schmit.
\newblock Characterization of chaotic quantum spectra and universality of level
  fluctuation laws.
\newblock {\em Physical Review Letters}, 52(1):1, 1984.

\bibitem{bordenave2012around}
Charles Bordenave and Djalil Chafa{\"\i}.
\newblock Around the circular law.
\newblock {\em Probability Surveys}, 9:1--89, 2012.

\bibitem{borodin2009ginibre}
Alexei Borodin and Christopher~D. Sinclair.
\newblock The {G}inibre ensemble of real random matrices and its scaling
  limits.
\newblock {\em Communications in Mathematical Physics}, 291:177--224, 2009.

\bibitem{bourgade2021extreme}
Paul Bourgade.
\newblock Extreme gaps between eigenvalues of {W}igner matrices.
\newblock {\em Journal of the European Mathematical Society}, 24(8):2823--2873,
  2021.

\bibitem{byun2022progress}
Sung-Soo Byun and Peter~J. Forrester.
\newblock Progress on the study of the {G}inibre ensembles {I}: {G}in{UE}.
\newblock {\em arXiv preprint arXiv:2211.16223}, 2022.

\bibitem{byun2023progress}
Sung-Soo Byun and Peter~J. Forrester.
\newblock Progress on the study of the {G}inibre ensembles {II}: {G}in{OE} and
  {G}in{SE}.
\newblock {\em arXiv preprint arXiv:2301.05022}, 2023.

\bibitem{deift1983ordinary}
Percy Deift, Tara Nanda, and Carlos Tomei.
\newblock Ordinary differential equations and the symmetric eigenvalue problem.
\newblock {\em SIAM Journal on Numerical Analysis}, 20(1):1--22, 1983.

\bibitem{feng2019small}
Renjie Feng, Gang Tian, and Dongyi Wei.
\newblock Small gaps of {GOE}.
\newblock {\em Geometric and Functional Analysis}, 29(6):1794--1827, 2019.

\bibitem{feng2021small}
Renjie Feng and Dongyi Wei.
\newblock Small gaps of circular $\beta$-ensemble.
\newblock {\em The Annals of Probability}, 49(2):997--1032, 2021.

\bibitem{figalli2016universality}
Alessio Figalli and Alice Guionnet.
\newblock Universality in several-matrix models via approximate transport maps.
\newblock {\em Acta Mathematica}, 217(1):115--159, 2016.

\bibitem{forrester2010log}
Peter~J. Forrester.
\newblock {\em Log-Gases and Random Matrices}.
\newblock Princeton University Press, 2010.

\bibitem{forrester2009method}
Peter~J. Forrester and Anthony Mays.
\newblock A method to calculate correlation functions for $\beta=1$ random
  matrices of odd size.
\newblock {\em Journal of Statistical Physics}, 134(3):443--462, 2009.

\bibitem{forrester2007eigenvalue}
Peter~J. Forrester and Taro Nagao.
\newblock Eigenvalue statistics of the real {G}inibre ensemble.
\newblock {\em Physical Review Letters}, 99(5):050603, 2007.

\bibitem{forrester2008skew}
Peter~J. Forrester and Taro Nagao.
\newblock Skew orthogonal polynomials and the partly symmetric real {G}inibre
  ensemble.
\newblock {\em Journal of Physics A: Mathematical and Theoretical},
  41(37):375003, 2008.

\bibitem{ge2017eigenvalue}
Stephen Ge.
\newblock {\em The eigenvalue spacing of {IID} random matrices and related
  least singular value results}.
\newblock {PhD} dissertation, UCLA, 2017.

\bibitem{gebert2019pure}
Martin Gebert and Mihail Poplavskyi.
\newblock On pure complex spectrum for truncations of random orthogonal
  matrices and {K}ac polynomials.
\newblock {\em arXiv preprint arXiv:1905.03154}, 2019.

\bibitem{goel2023central}
Advay Goel, Patrick Lopatto, and Xiaoyu Xie.
\newblock Central limit theorem for the complex eigenvalues of {G}aussian
  random matrices.
\newblock {\em arXiv preprint arXiv:2306.10243}, 2023.

\bibitem{horn2012matrix}
Roger Horn and Charles Johnson.
\newblock {\em Matrix Analysis}.
\newblock Cambridge University Press, 2012.

\bibitem{kallenberg2017random}
Olav Kallenberg.
\newblock {\em Random Measures, Theory and Applications}.
\newblock Springer, 2017.

\bibitem{kopel2015linear}
Phil Kopel.
\newblock Linear statistics of non-{H}ermitian matrices matching the real or
  complex {G}inibre ensemble to four moments.
\newblock {\em arXiv preprint arXiv:1510.02987}, 2015.

\bibitem{kriecherbauer2008locating}
T.~Kriecherbauer, A.B.J. Kuijlaars, K.D. T.-R. McLaughlin, and P.D. Miller.
\newblock Locating the zeros of partial sums of $e^z$ with {R}iemann--{H}ilbert
  methods.
\newblock In {\em Integrable Systems and Random Matrices: In Honor of Percy
  Deift: Conference on Integrable Systems, Random Matrices, and Applications in
  Honor of Percy Deift's 60th Birthday, May 22-26, 2006, Courant Institute of
  Mathematical Sciences, New York University, New York}, volume 458, page 183.
  American Mathematical Soc., 2008.

\bibitem{landon2020comparison}
Benjamin Landon, Patrick Lopatto, and Jake Marcinek.
\newblock Comparison theorem for some extremal eigenvalue statistics.
\newblock {\em The Annals of Probability}, 48(6):2894--2919, 2020.

\bibitem{lopatto2021tail}
Patrick Lopatto and Kyle Luh.
\newblock Tail bounds for gaps between eigenvalues of sparse random matrices.
\newblock {\em Electronic Journal of Probability}, 26:1--26, 2021.

\bibitem{luh2021eigenvectors}
Kyle Luh and Sean O’Rourke.
\newblock Eigenvectors and controllability of non-{H}ermitian random matrices
  and directed graphs.
\newblock {\em Electronic Journal of Probability}, 26, 2021.

\bibitem{mays2012geometrical}
Anthony Mays.
\newblock A geometrical triumvirate of real random matrices.
\newblock {\em arXiv preprint arXiv:1202.1218}, 2012.

\bibitem{nguyen2017random}
Hoi Nguyen, Terence Tao, and Van Vu.
\newblock Random matrices: tail bounds for gaps between eigenvalues.
\newblock {\em Probability Theory and Related Fields}, 167:777--816, 2017.

\bibitem{odlyzko1987distribution}
Andrew~M. Odlyzko.
\newblock On the distribution of spacings between zeros of the zeta function.
\newblock {\em Mathematics of Computation}, 48(177):273--308, 1987.

\bibitem{shi2012smallest}
Dai Shi and Yunjiang Jiang.
\newblock Smallest gaps between eigenvalues of random matrices with complex
  {G}inibre, {W}ishart and universal unitary ensembles.
\newblock {\em arXiv preprint arXiv:1207.4240}, 2012.

\bibitem{sinclair2009correlation}
Christopher~D. Sinclair.
\newblock Correlation functions for $\beta=1$ ensembles of matrices of odd
  size.
\newblock {\em Journal of Statistical Physics}, 136(1):17--33, 2009.

\bibitem{sommers2008general}
Hans-J{\"u}rgen Sommers and Waldemar Wieczorek.
\newblock General eigenvalue correlations for the real {G}inibre ensemble.
\newblock {\em Journal of Physics A: Mathematical and Theoretical},
  41(40):405003, 2008.

\bibitem{soshnikov1998level}
Alexander Soshnikov.
\newblock Level spacings distribution for large random matrices: Gaussian
  fluctuations.
\newblock {\em Annals of Mathematics}, pages 573--617, 1998.

\bibitem{soshnikov2000determinantal}
Alexander Soshnikov.
\newblock Determinantal random point fields.
\newblock {\em Russian Mathematical Surveys}, 55(5):923, 2000.

\bibitem{soshnikov2005statistics}
Alexander Soshnikov.
\newblock Statistics of extreme spacing in determinantal random point
  processes.
\newblock {\em Moscow Mathematical Journal}, 5(3):705--719, 2005.

\bibitem{thoma2022overcrowding}
Eric Thoma.
\newblock Overcrowding and separation estimates for the {C}oulomb gas.
\newblock {\em Communications on Pure and Applied Mathematics}, 2022.

\bibitem{vinson2001closest}
Jade~P. Vinson.
\newblock {\em Closest Spacing of Eigenvalues}.
\newblock {PhD} dissertation, Princeton University, 2001.
\newblock arXiv:1111.2743.

\end{thebibliography}
\end{document}